\documentclass[11pt,reqno]{amsart}

\makeatletter
\usepackage{amssymb}
\usepackage{amssymb,amsmath,color,comment}
\usepackage{amsbsy}
\usepackage{amsfonts}

\def\marginpar#1{\ignorespaces}

\newtheorem{theorem}[equation]{Theorem}
\newtheorem{proposition}[equation]{Proposition}
\newtheorem{lemma}[equation]{Lemma}
\newtheorem{corollary}[equation]{Corollary}
\newtheorem{definition}[equation]{Definition}

\newtheorem{conjecture}[equation]{Conjecture}

\theoremstyle{definition}
\newtheorem{remark}[equation]{Remark}

\numberwithin{equation}{section}

\def\AArm{\fam0 \rm}%
\newdimen\AAdi%
\newbox\AAbo%
\def\AAk#1#2{\setbox\AAbo=\hbox{#2}\AAdi=\wd\AAbo\kern#1\AAdi{}}%

\newcommand{\BBone}{{\ensuremath{{\AArm 1\AAk{-.8}{I}I}}}}

\def\eqref#1{(\ref{#1})}
\def\eqlabel#1{\def\@currentlabel{#1}}

\def\formula#1{\def\@tempa{#1}\let\@tempb\theequation\def\theequation{%
\hbox{#1}}\def\@currentlabel{(\theequation)}$$}
\def\endformula{\leqno\hbox{(\@tempa)}$$\@ignoretrue\let\theequation\@tempb}

\def\given{\hskip5\p@\relax\vrule\@width.4\p@\hskip5\p@\relax}

\newcommand{\open}[1]{%
\par\normalfont\topsep6\p@\@plus6\p@\trivlist\item[\hskip\labelsep\itshape#1%
\@addpunct{.}]\ignorespaces}

\DeclareRobustCommand{\close}[1]{%
  \ifmmode 
  \else \leavevmode\unskip\penalty9999 \hbox{}\nobreak\hfill
  \fi
  \quad\hbox{$#1$}}

\newlength{\toskip}

\def\<{\langle}
\def\>{\rangle}

\def \R {{\mathbb R}}
\def \Q {{\mathbb Q}}

\def \G {{\mathbb G}}

\def \E {{\mathbb E}}

\def \L {{\mathbb L}}

\def \Var {\textrm{Var}}

\def \Osc {\textrm{Osc}}

\def \Cov {\textrm{Cov}}
\makeatother

\begin{document}
\date{\today}

\title[KLS]{On the Poincar\'e constant of log-concave measures.}

\author[P. Cattiaux]{\textbf{\quad {Patrick} Cattiaux $^{\spadesuit}$ \, \, }}
\address{{\bf {Patrick} CATTIAUX},\\ Institut de Math\'ematiques de Toulouse. Universit\'e de Toulouse, CNRS UMR 5219. \\ 118 route de Narbonne, F-31062 Toulouse cedex 09.}
\email{cattiaux@math.univ-toulouse.fr}

\author[A. Guillin]{\textbf{\quad {Arnaud} Guillin $^{\diamondsuit}$}}
\address{{\bf {Arnaud} GUILLIN},\\ Laboratoire de Math\'ematiques, CNRS UMR 6620, Universit\'e Blaise Pascal.\\ avenue des Landais, F-63177 Aubi\`ere.} \email{guillin@math.univ-bpclermont.fr}

\maketitle
 \begin{center}

 \textsc{$^{\spadesuit}$  Universit\'e de Toulouse}
\smallskip

\textsc{$^{\diamondsuit}$Universit\'e Blaise Pascal}
\smallskip

 \end{center}

\begin{abstract}
The goal of this paper is to push forward the study of those properties of log-concave measures that help to estimate their Poincar\'e constant. First we revisit E. Milman's result \cite{Emil1} on the link between weak  (Poincar\'e or concentration) inequalities and Cheeger's inequality in the logconcave cases, in particular extending localization ideas and a result of Latala, as well as providing a simpler proof of the nice Poincar\'e (dimensional) bound in the inconditional case. Then we prove alternative transference principle by concentration or using various distances (total variation, Wasserstein). A mollification procedure is also introduced enabling, in the logconcave case, to reduce to the case of the Poincar\'e inequality for the mollified measure. We finally complete the transference section by the comparison of various probability metrics (Fortet-Mourier, bounded-Lipschitz,...).
\end{abstract}
\bigskip

\bigskip

\textit{ MSC 2000 : .}
\bigskip

\section{\bf Introduction and overview.}\label{Intro}

In the whole paper, for $x\in \mathbb R^n$, $|x|=\left(\sum_{i=1}^n x_i^2\right)^{\frac 12}$ denotes the euclidean norm of $x$, and a function $f$ is said to be $K$-Lipschitz if $\sup_{|x-y|>0} \frac{|f(x)-f(y)|}{|x-y|} \leq K$.
\medskip

Let $\nu$ be a Probability measure defined on $\R^n$. For a real valued function $f$, $\nu(f)$ and $m_\nu(f)$ will denote respectively the $\nu$ mean and a $\nu$ median of $f$, when these quantities exist. We also denote by $$\Var_\nu(f)=\nu(f^2) - \nu^2(f)$$ the $\nu$ variance of $f$. \\ The Poincar\'e constant $C_P(\nu)$ of $\nu$ is defined as the best constant such that $$\Var_\nu(f)\leq C_P(\nu) \, \nu(|\nabla f|^2) \, .$$ In all the paper, we shall denote equally $C_P(Z)$ or $C_P(\nu)$ the Poincar\'e constant for a random variable $Z$ with distribution $\nu$. We say that $\nu$ satisfies a Poincar\'e inequality when $C_P(\nu)$ is finite. \\ It is well known that, as soon as a Poincar\'e inequality is satisfied, the tails of $\nu$ are exponentially small, i.e. $\nu(|x|>R)\leq C \, e^{- \, cR/\sqrt{C_P(\nu)}}$ for some universal $c$ and $C$ (see \cite{BL97}), giving a very useful necessary condition for this inequality to hold. Conversely, during the last eighty years a lot of sufficient conditions have been given for a Poincar\'e inequality to hold. In 1976, Brascamp and Lieb (\cite{BrasLieb}) connected Poincar\'e inequality to convexity by proving the following: if $\nu(dx)=e^{-V(x)} \, dx$, then $$ \Var_\nu(f) \leq \int \, ^t \, \nabla f \, (Hess^{-1}(V)) \, \nabla f \, d\nu$$ where $Hess(V)$ denotes the Hessian matrix of $V$. Consequently, if $V$ is uniformly convex, i.e. $\inf_x \, ^t\xi \, Hess(V)(x) \xi  \geq \rho \, |\xi|^2$ for some $\rho >0$, then $C_P(\nu) \leq 1/\rho$. This result contains in particular the gaussian case, and actually gaussian measures achieve the Brascamp-Lieb bound as it is easily seen looking at linear functions $f$. \\ This result was extended to much more general ``uniformly convex'' situations through the celebrated $\Gamma_2$ theory introduced by Bakry and Emery (see the recent monograph \cite{BGL} for an up to date state of the art of the theory) and the particular uniformly convex situation corresponds to the CD($\rho,\infty)$ curvature-dimension property in this theory. This theory has been recently revisited in \cite{CGsp} by using coupling techniques for the underlying stochastic process. 
\medskip

A particular property of the Poincar\'e inequality is the tensorization property $$C_P(\nu_1\otimes ... \otimes \nu_N) \leq \max_{i=1,...,N} \, C_P(\nu_i) \, .$$  It is of fundamental importance for the concentration of measure and for getting bounds for functionals of independent samples in statistics, due to its ``dimension free'' character. This ``dimension free'' character is captured by the Brascamp-Lieb inequality or the Bakry-Emery criterion, even for non-product measures.

Using a simple perturbation of $V$ by adding a bounded (or a Lipschitz) term, one can show that uniform convexity ``at infinity'' is enough to get a Poincar\'e inequality. This result can also be proved by using reflection coupling (see \cite{Ebe,Ebe2,CGsp}). However in this situation a ``dimension free'' bound for the optimal constant is hard to obtain, as it is well known for the double well potential $V(x) = |x|^4 - |x|^2$. 
\medskip

Uniform convexity (even at infinity) is not necessary as shown by the example of the symmetric exponential measure on the line, $\nu(dx)=\frac 12 \, e^{-|x|}$ which satisfies $C_P(\nu)=4$. In 1999, Bobkov (\cite{bob99}) has shown that any log-concave probability measure satisfies the Poincar\'e inequality. Here log-concave means that $\nu(dx)=e^{-V(x)} \, dx$ where $V$ is a convex function with values in $\mathbb R \cup \{+ \infty\}$. In particular uniform measures on convex bodies are log-concave. We refer to the recent book \cite{Gia} for an overview on the topic of convex bodies, and to \cite{SW} for a survey of log-concavity in statistics. Another proof, applying to a larger class of measures, was given in \cite{BBCG} using Lyapunov functions as introduced in \cite{BCG}. If it is now known that a Poincar\'e inequality is equivalent to the existence of a Lyapunov function (see \cite{CGZ,CGjfa}), this approach is far to give good controls for the Poincar\'e constant.

Actually Bobkov's result is stronger since it deals with the $\L^1$ version of Poincar\'e inequality 
\begin{equation}\label{eqcheegbob}
\nu(|f-\nu(f)|) \leq C_C(\nu) \, \nu(|\nabla f|) \, ,
\end{equation}
which is often called Cheeger inequality. Another form of Cheeger inequality is 
\begin{equation}\label{eqcheegbob1}
\nu(|f-m_\nu(f)|) \leq C'_C(\nu) \, \nu(|\nabla f|) \, .
\end{equation}
Using $$\frac 12 \, \nu(|f-\nu(f)|) \, \leq \, \nu(|f-m_\nu(f)|) \, \leq \, \nu(|f-\nu(f)|) \, ,$$ it immediately follows that $\frac 12 \, C_C \leq C'_C \leq C_C$.\\ It is well known that the Cheeger constant gives a natural control for the isoperimetric function $Is_\nu$ of $\nu$. Recall that for $0\leq u \leq \frac 12$, $$Is_\nu(u) = \inf_{A, \nu(A)=u} \, \nu_s(\partial A)$$ where $\nu_s(\partial A)$ denotes the surface measure of the boundary of $A$. It can be shown that $$Is_\nu(u) = \frac{u}{C'_C(\nu)} \, \geq \frac{u}{C_C(\nu)} \,.$$ The Cheeger inequality is stronger than the Poincar\'e inequality and 
\begin{equation}\label{eqcheegpoinc}
C_P(\nu)\leq 4 \, (C'_C)^2(\nu) \, .
\end{equation}
The first remarkable fact in the log-concave situation is that a converse inequality holds, namely if $\nu$ is log-concave, 
\begin{equation}\label{eqled}
(C'_C)^2(\nu) \leq 36 \, C_P(\nu) \, , 
\end{equation}
as shown by Ledoux (\cite{ledgap} formula (5.8)). \\ Ledoux's approach is using the associated semi-goup with generator $L=\Delta - \nabla V. \nabla$ for which the usual terminology corresponding to the convexity of $V$ is zero curvature. Of course to define $L$ one has to assume some smoothness of $V$. But if $Z$ is a random variable with a log-concave distribution $\nu$ and $G$ is an independent standard gaussian random variable, Prekopa-Leindler theorem ensures that the distribution of $Z+\varepsilon \, G$ is still log concave for any $\varepsilon >0$ and is associated to a smooth potential $V_\varepsilon$. Hence we may always assume that the potential $V$ is smooth provided we may pass to the limit $\varepsilon \to 0$.
\medskip

Log-concave measures deserve attention during the last twenty years in particular in Statistics. They are considered close to product measures in high dimension. It is thus important to get some tractable bound for their Poincar\'e constant, in particular to understand the role of the dimension. 

Of course if $Z=(Z_1,...,Z_n)$ is a random vector of $\mathbb R^n$, $$C_P((\lambda_1 \, Z_1,...,\lambda_n \, Z_n)) \leq \max_i \, \lambda_i^2 \, C_P((Z_1,...,Z_n)) \, ,$$ and the Poincar\'e constant is unchanged if we perform a translation or an isometric change of coordinates. It follows that $$C_P(Z) \, \leq  \, \sigma^2(Z) \, C_P(Z') \,$$ where $\sigma^2(Z)$ denotes the largest eigenvalue of the covariance matrix $Cov_{i,j}(Z) = \Cov(Z_i, Z_j)$, and $Z'$ is an affine transformation of $Z$ which is centered and with Covariance matrix equal to Identity. Such a random vector (or its distribution) is called \underline{isotropic} (or in isotropic position for convex bodies and their uniform distribution). The reader has to take care about the use of the word isotropic, which has a different meaning in probability theory (for instance in Paul L\'evy's work).

Applying the Poincar\'e inequality to linear functions show that $\sigma^2(Z) \leq C_P(Z)$.  In particular, in the uniformly convex situation, $\sigma^2(Z) \leq 1/\rho$ with equality when $Z$ is a gaussian random vector. For the symmetric exponential measure on $\mathbb R$, we also have $\sigma^2(Z)=C_P(Z)$ while $\rho=0$. It thus seems plausible that, even in positive curvature, $\sigma^2(Z)$ is the good parameter to control the Poincar\'e constant. 

The following was conjectured by Kannan, Lov\'{a}sz and Simonovits (\cite{KLS})
\begin{conjecture}\label{conjKLS}{\it (K-L-S conjecture.)}
There exists a universal constant $C$ such that for any log-concave probability measure $\nu$ on
$\R^n$, $$C_P(\nu) \, \leq \, C \, \sigma^2(\nu) \, ,$$ where $\sigma^2(\nu)$ denotes the largest
eigenvalue of the covariance matrix $Cov_{i,j}(\nu) = \Cov_\nu(x_i, x_j)$, or if one prefers, there exists an universal constant $C$ such that any isotropic log-concave probability measure $\nu$, in any dimension $n$, satisfies $C_P(\nu)\leq C$.
\end{conjecture}

During the last years a lot of works have been done on this conjecture. A recent book \cite{Alonbast} is totally devoted to the description of the state of the art. We will thus mainly refer to this book for references, but of course apologize to all important contributors. We shall just mention part of these works we shall revisit and extend.
\medskip

In this note we shall on one hand investigate properties of log-concave measures that help to evaluate their Poincar\'e constant, on the other hand obtain explicit constants in many intermediate results. Let us explain on an example: in a remarkable paper (\cite{Emil1}), E. Milman has shown that one obtains an equivalent inequality if one replaces the energy of the gradient in the right hand side of the Poincar\'e inequality by the square of its Lipschitz norm, furnishing a much less demanding inequality ($(2,+\infty)$ Poincar\'e inequality). The corresponding constant is sometimes called the spread constant. In other words the Poincar\'e constant of a log-concave measure $\nu$ is controlled by its spread constant. In the next section we shall give another proof of this result. Actually we shall extend it to weak forms of $(1,+\infty)$ inequalities. These weak forms allow us to directly compare the concentration profile of $\nu$ with the corresponding weak inequality. We shall also give explicit controls of the constants when one reduces the support of $\nu$ to an euclidean ball as in \cite{Emil1} or a $l^\infty$ ball, the latter being an explicit form of a result by Latala \cite{Latala}. In section \ref{sectransfer} we shall describe several transference results using absolute continuity, concentration properties and distances between measures. The novelty here is that we compare a log-concave measure $\nu$ with another non necessarily log-concave measure $\mu$. For instance, we show that if the distance between $\nu$ and $\mu$ is small enough, then the Poincar\'e constant of $\mu$ controls the one of $\nu$. This is shown for several distances: total variation, Wasserstein, Bounded Lipschitz. Section \ref{molly} is concerned with mollification. The first part is a revisit of results by Klartag \cite{Klarcube}. The second part studies convolution with a gaussian kernel. It is shown that if $\gamma$ is some gaussian measure, $C_P(\nu*\gamma)$ controls $C_P(\nu)$. The proof is based on stochastic calculus. Finally in the last section and using what precedes we show that all the previous distances and the L\'evy-Prokhorov distance define the same uniform structure on the set of log-concave measures independently of the dimension. We thus complete the transference results using distances. Some dimensional comparisons have been done in \cite{Meckes2}.
\bigskip

\section{\bf Revisiting E. Milman's results.}\label{secEmil}

\subsection{$(p,q)$ Poincar\'e inequalities. \\ \\}\label{subsecpq}

Following \cite{Emil1}, the usual Poincar\'e inequality can be generalized
in a $(p,q)$ Poincar\'e inequality, for $1\leq p \leq q \leq +\infty$,
\begin{equation}\label{eqpoincpq}
B_{p,q} \, \nu^{1/p}(|f-\nu(f)|^p) \leq \nu^{1/q}(|\nabla f|^q) .
\end{equation}
\smallskip
For $p=q=2$ we recognize the Poincar\'e inequality and $B^2_{2,2}=1/C_P(\nu)$, and for $p=q=1$ the Cheeger inequality with $B_{1,1}=1/C_C(\nu)$. 

Among all $(p,q)$ Poincar\'e inequalities, the weakest one is clearly the $(1,+\infty)$ one, the strongest
the $(1,1)$ one, we called Cheeger's inequality previously. Indeed for $1\leq p \leq p' \leq q \leq +\infty$ except the case $p=q=+\infty$, one has the
following schematic array between these Poincar\'e inequalities
\begin{center}
$$
\begin{array}{cccc}
& (1,1) & \Rightarrow & (1,q) \\
& \Downarrow & & \Uparrow \\
& (p,p) & \Rightarrow & (p,q)\\
& \Downarrow & & \Uparrow \\
& (p',p') & \Rightarrow & (p',q)\\
\end{array}
$$
\end{center}
The meaning of all these inequalities is however quite unclear except some cases we shall describe below. 

First remark that on $\R^n$,
$$|f(x)-f(a)| \leq
\parallel \nabla f\parallel_\infty \, |x-a|$$
yielding $$\nu(|f-m_\nu(f)|) = \inf_b \, \nu(|f-b|) \leq \inf_a \, \nu^{1/p}(|x-a|^p) \, \parallel |\nabla f|\parallel_\infty \, , $$ so that since
\begin{equation}\label{eqcompare}
\frac 12 \, \nu(|f-\nu(f)|) \, \leq \, \nu(|f-m_\nu(f)|) \, \leq \, \nu(|f-\nu(f)|) \, ,
\end{equation}
the $(p,+\infty)$ Poincar\'e inequality is satisfied as soon as $\nu$ admits a $p$-moment. There is thus no hope for this inequality to be helpful unless we make some additional assumption.
\smallskip

Now look at the $(p,2)$ Poincar\'e inequality ($1\leq p \leq 2$). We may write, assuming that $\nu(f)=0$,
\begin{eqnarray*}
\Var_\nu(f) & \leq & \parallel f\parallel_\infty^{2-p} \, \nu(|f|^p) \leq  \frac{1}{B^p_{p,2}} \, \parallel f\parallel_\infty^{2-p} \,
\nu^{p/2}(|\nabla f|^2)
\end{eqnarray*}
which is equivalent to $$\Var_\nu(f) \, \leq   \, c \, s^{- \, \frac{2-p}{p}} \, \nu(|\nabla f|^2) + s \, \parallel f-\nu(f)\parallel_\infty^{2} \quad \textrm{ for all } s>0 \, ,$$ with $$\frac{1}{B^2_{p,2}} =
\frac 1p \, \left(\frac{(c(2-p))^{2/p} + p^{2/p}}{(c(2-p))^{(2-p)/p}}\right) \, .$$ This kind of
inequalities has been studied under the name of weak Poincar\'e inequalities (see
\cite{rw,BCR2,CGGR}). They can be used to show the $\L^2-\L^\infty$ convergence of the semi-group
with a rate $t^{- p/(2-p)}$.\\ As shown in \cite{rw}, any probability measure $\nu(dx)=e^{-V(x)} \, dx$ such that $V$ is locally bounded, satisfies a weak Poincar\'e inequality. Indeed, using Holley-Stroock perturbation argument w.r.t. the (normalized) uniform measure on the euclidean ball $B(0,R)$, it is easy to see that   $$\Var_\nu(f) \leq \frac{4R^2}{\pi^2} \, e^{Osc_R \, V} \, \nu(|\nabla f|^2) + 2 \, \nu(|x|>R) \, \parallel f-\nu(f)\parallel_\infty^2$$ where $Osc_R \, V = \sup_{|x|\leq R} V(x) - \inf_{|x|\leq R} V(x)$.

It is thus tempting to introduce weak versions of $(p,q)$ Poincar\'e inequalities. \begin{definition}\label{defweakinfty}
We shall say that $\nu$ satisfies a weak $(p,q)$ Poincar\'e inequality if there exists some non increasing non-negative function $\beta$ defined on $]0,+\infty[$ such that for all $s>0$ and all smooth function $f$, $$(\nu(|f-\nu(f)|^p))^\frac 1p \leq \beta(s) \, \parallel |\nabla f|\parallel_q + s \, \Osc( f) \, ,$$
where $\Osc$ denotes the oscillation of $f$.  We shall sometimes replace $\nu(f)$ by $m_\nu(f)$. In particular for $p=1$ and $q=\infty$ we have $$\beta^{med}(s) \leq \beta^{mean}(s) \leq 2 \beta^{med}(s/2) \, .$$
\end{definition}
Of course if $\beta(0)<+\infty$ we recover the $(p,q)$ Poincar\'e inequality. Any probability measure satisfies a peculiar weak $(1,+\infty)$ Poincar\'e inequality, namely
\begin{proposition}\label{propwinftyconc}
Denote by $\alpha_\nu$ the concentration profile of a probability measure $\nu$, i.e. $$\alpha_\nu(r) := \sup \{1 - \, \nu(A+B(0,r)) \, ; \, \nu(A) \geq \frac 12\} \, , \, r>0 \, ,$$ where $B(y,r)$ denotes the euclidean ball centered at $y$ with radius $r$. Then for any probability measure $\nu$ and all $s>0$, $$\nu(|f-m_\nu(f)|) \leq \, \alpha_\nu^{-1}(s/2) \, \parallel |\nabla f|\parallel_\infty + s \, \Osc f \, .$$ and $$\nu(|f-\nu(f)|) \leq 2 \, \alpha_\nu^{-1}(s/4) \, \parallel |\nabla f|\parallel_\infty + s \, \Osc (f) \, ,$$ where $\alpha_\nu^{-1}$ denotes the converse function of $\alpha_\nu$.
\end{proposition}
\begin{proof}
Due to homogeneity we may assume that $f$ is $1$-Lipschitz. Hence $\nu(|f-m_\nu(f)|>r) \leq 2 \, \alpha_\nu(r)$. Thus 
\begin{eqnarray*}
\nu(|f-\nu(f)|) &\leq& 2 \, \nu(|f-m_\nu(f)|) \leq 2 \, r \, \nu(|f-m_\nu(f)|\leq r) + 2 \, \Osc (f) \, \nu(|f-m_\nu(f)|>r) \\ &\leq& 2r + 4 \, \Osc (f) \, \alpha_\nu(r) \, ,
\end{eqnarray*}
hence the result.
\end{proof}
This trivial result will be surprisingly useful for the family of log-concave measures. Recall the following result is due to E. Milman (see Theorem 2.4 in \cite{Emil1})

\begin{theorem}\label{thmmainemil}[E. Milman's theorem.]
If $d\nu=e^{-V} dx$ is a log-concave probability measure in $\R^n$, there exists a universal
constant $C$ such that for all $(p,q)$ and $(p',q')$ (with $1\leq p \leq q \leq +\infty$ and
$1\leq p' \leq q' \leq +\infty$) $$B_{p,q} \leq C \, p' \, B_{p',q'} \, .$$
\end{theorem}

Hence in the log-concave situation the $(1,+\infty)$ Poincar\'e inequality implies Cheeger's
inequality, more precisely implies a control on the Cheeger's constant (that any log-concave
probability satisfies a Cheeger's inequality is already well known). We shall revisit this last result in the next subsection.
\bigskip

\subsection{Log-concave Probability measures. \\ \\}\label{subsecconcave}

In order to prove Theorem \ref{thmmainemil} it is enough to show that the $(1,+\infty)$ Poincar\'e
inequality implies the $(1,1)$ one, and to use the previous array. We shall below reinforce E. Milman's result. The proof (as in \cite{Emil1}) lies on the concavity of the isoperimetric profile, namely the following proposition which was obtained by several authors (see \cite{Emil1} Theorem 1.8 for a list)
\begin{proposition}\label{propconcav}
Let $\nu$ be a (smooth) log-concave probability measure on $\mathbb R^n$. Then the isoperimetric profile $u \mapsto Is_\nu(u)$ is concave on $[0,\frac 12]$.
\end{proposition}
The previous concavity assumption may be used to get some estimates on Poincar\'e and Cheeger constants
\begin{proposition}\label{thmconcprofile}
Let $\nu$ be a probability measure such that $u \mapsto Is_\nu(u)$ is concave on $[0,\frac 12]$. Assume that there exist some $0\leq u \leq 1/2$ and some $C(u)$ such that for any Lipschitz function $f\geq 0$ it holds 
\begin{equation}\label{eqconcprofosc}
\nu(|f-\nu(f)|) \leq C(u) \, \nu(|\nabla f|) \, + \, u \, Osc (f) \, .
\end{equation}
Then for all measurable $A$ such that $\nu(A)\leq 1/2$, $$\nu_s(\partial A) \geq \frac{1 \, - \, 2u}{C(u)} \, \nu(A) \quad \textrm{ i.e.} \quad C'_C(\nu) \leq \frac{ C(u)}{ 1 \, - \, 2u} \, .$$ If we reinforce \eqref{eqconcprofosc} as follows, for some  $0\leq u \leq 1$, $1<p<\infty$ and some $C_p(u)$
\begin{equation}\label{eqconcprofvar}
\nu(|f-\nu(f)|) \leq C_p(u) \, \nu(|\nabla f|) \, + \, u \, \left(\int (f-\nu(f))^pd\nu\right)^{\frac 1p} \, ,
\end{equation}
then for all measurable $A$ such that $\nu(A)\leq 1/2$, $$\nu_s(\partial A) \geq \frac{ 1 \, - \, u}{C_p(u)} \, \nu(A) \quad \textrm{ i.e.} \quad C'_C(\nu) \leq \frac{C_p(u)}{ 1 \, - \, u} \, .$$
\end{proposition}
\begin{proof}
Let $A$ be some Borel subset with $\nu(A)=\frac 12$. According to Lemma 3.5 in \cite{BH} one can find a sequence $f_n$ of Lipschitz functions with $0\leq f_n\leq 1$, such that $f_n \to \mathbf 1_{\bar A}$ pointwise ($\bar A$ being the closure of $A$) and $\limsup \nu(|\nabla f_n|) \leq \nu_s(\partial A)$. According to the proof of Lemma 3.5 in \cite{BH}, we may assume that $\nu(\bar A)=\nu(A)$ (otherwise $\nu_s(\partial A)=+\infty$). Taking limits in the left hand side of \eqref{eqconcprofosc} thanks to Lebesgue's bounded convergence theorem, we thus obtain $$\nu(|\mathbf 1_A - \nu(A)|) \leq C(u) \, \nu_s(\partial A) + u \, .$$ The left hand side is equal to $2 \nu(A) (1-\nu(A))=\frac 12$ so that we obtain $\nu_s(\partial A) \geq \frac{\frac 12 -u}{C(u)}$. It remains to use the concavity of $Is_\nu$, which yields $Is_\nu(u) \geq 2 \, Is_\nu(\frac 12) \, u$. \\ If we replace \eqref{eqconcprofosc} by \eqref{eqconcprofvar}, we similarly obtain, when $\nu(A)= \frac 12$, $\int(1_A-\nu(A))^pd\nu=(1/2)^p$, so that $\frac 12 \leq C_p(u) \, \nu_s(\partial A) + \frac 12 \, u$ and the result follows similarly.
\end{proof} 
\begin{remark}\label{rem1}
We may replace $\nu(f)$ by $m_\nu(f)$ in \eqref{eqconcprofosc} without changing the proof, since the explicit form of the approximating $f_n$ in \cite{BH} satisfies $m_\nu(f_n) \to \frac 12$. \hfill $\diamondsuit$
\end{remark}
According to Proposition \ref{propconcav}, the previous proposition applies to log-concave measures. But in this case one can weaken the required inequalities
\begin{theorem}\label{thmweakmil}
Let $\nu$ a log-concave probability measure. \\ Assume that there exist some $0\leq s < 1/2$ and some $\beta(s)$ such that for any Lipschitz function $f$ it holds
\begin{equation}\label{eqlogconcosc}
\nu(|f-\nu(f)|) \leq \beta(s) \, \parallel|\nabla f|\parallel_\infty \, + \, s \, Osc (f) \, ,
\end{equation}
respectively, for some $0\leq s<1$ and some $\beta(s)$ 
\begin{equation}\label{eqlogconcvar}
\nu(|f-\nu(f)|) \leq \beta(s) \, \parallel|\nabla f|\parallel_\infty \, + \, s \, \left(\Var_\nu(f)\right)^{\frac 12} \, .
\end{equation}
Then $$ C'_C(\nu) \leq \frac{4 \beta(s)}{\pi \, (\frac 12 -s)^2} \quad \textrm{ resp. } \quad C'_C(\nu) \leq \frac{16 \beta(s)}{ \pi \, (1-s)^2} \, .$$
We may replace $\nu(f)$ by $m_\nu(f)$ in both cases.
\end{theorem}
\begin{proof}
In the sequel $P_t$ denotes the symmetric semi-group with infinitesimal generator $L=\Delta - \nabla V.\nabla$. Here we assume for simplicity that $V$ is smooth on the interior of $D=\{V<+\infty\}$ (which is open and convex), so that the generator acts on functions whose normal derivative on $\partial D$ is equal to 0. From the probabilistic point of view, $P_t$ is associated to the diffusion process with generator $L$ normally reflected at the boundary $\partial D$.

A first application of zero curvature is the following, that holds for all $t>0$, 
\begin{equation}\label{eqledgap2}
\parallel
|\nabla P_tg|\parallel_\infty \, \leq \, \frac{1}{\sqrt{\pi t}} \, \parallel g\parallel_\infty \, .
\end{equation}
This result is proved using reflection coupling in Proposition 17 of \cite{CGsp}. With the slightly worse constant $\sqrt{2 t}$, it was previously obtained by Ledoux in \cite{ledgap}. According to Ledoux's duality argument (see (5.5) in \cite{ledgap}), we deduce, with $g=f-\nu(f)$,
\begin{equation}\label{eqledgap}
\nu(|g|) \leq \sqrt{4t/\pi} \, \nu(|\nabla f|) + \nu(|P_t g|) \, .
\end{equation}
Note that $\nu(|P_t g|)=\nu(|P_tf -\nu(P_tf)|)$. Applying \eqref{eqlogconcosc} with $P_tf$, we obtain $$\nu(|f-\nu(f)|) \leq \sqrt{4t/\pi} \, \nu(|\nabla f|) +  \beta(s) \, \parallel 
|\nabla P_tf|\parallel_\infty \, + s \, Osc (P_tf) \, .$$ Applying \eqref{eqledgap2} again, and the contraction property of the semi-group in $\mathbb L^\infty$, yielding $Osc(P_tf)\leq Osc (f)$, we get 
\begin{equation}\label{eqzeroweakcheeg}
\nu(|f-\nu(f)|) \leq \sqrt{4t/\pi} \, \nu(|\nabla f|) + \left(s \, + \frac{\beta(s)}{\sqrt{\pi t}}\right) \,  Osc (f) \, .
\end{equation}
Choose $t =  \frac{4 \, \beta^2(s)}{\pi \, (\frac 12 -s)^2}$. We may apply proposition \ref{eqconcprofosc} (and remark \ref{rem1}) with $u=(s+\frac 12)/2$ which is less than $\frac 12$ and $$C(u)=\frac{4 \beta(s)}{\pi \, (\frac 12 -s)} \, .$$ yielding the result.\\ If we want to deal with the case of $m_\nu(f)$ we have to slightly modify the proof. This time we choose $g=f-m_\nu(P_tf)$ so that, first $\nu(|f-m_\nu(f)|) \leq \nu(|g|)$, second $P_tg=P_tf - m_\nu(P_tf)$, so that we can apply \eqref{eqlogconcosc} with the median. We are done by using remark \ref{rem1}.
\smallskip

Next we apply \eqref{eqlogconcvar} and the contraction property of the semi-group in $\mathbb L^2$. We get $$\nu(|f-\nu(f)|) \leq \sqrt{4t/\pi} \, \nu(|\nabla f|) +  \beta(s) \, \parallel 
|\nabla P_tf|\parallel_\infty \, + s \, \left(\Var_\nu(f)\right)^{\frac 12} \, .$$ But now, either $\Var_\nu(f) \leq \frac 14 \, Osc(f)$ or $\Var_\nu(f) \geq \frac 14 \, Osc(f)$. \\ In the first case we get  $$\nu(|f-\nu(f)|) \leq \sqrt{4t/\pi} \, \nu(|\nabla f|) +  \beta(s) \, \parallel 
|\nabla P_tf|\parallel_\infty \, + \, \frac s2 \, Osc(f) \, ,$$ and as we did before we finally get, for 
\begin{equation}\label{eqcasvar1}
\nu(|f-\nu(f)|) \leq  \, \frac{8 \, \beta(s)}{\pi \, (1-s)} \, \nu(|\nabla f|) +   \, \frac{s+1}{4} \, Osc(f) \, .
\end{equation}
One can notice that $\frac{s+1}{4} < \frac 12$.\\
In the second case, we first have $$\parallel |\nabla P_tf|\parallel_\infty \, \leq \, \frac{2}{\sqrt{\pi t}} \, \left(\Var_\nu(f)\right)^{\frac 12} \, ,$$ so that finally
\begin{equation}\label{eqcasvar2}
\nu(|f-\nu(f)|) \leq  \, \frac{8 \, \beta(s)}{\pi \, (1-s)} \, \nu(|\nabla f|) +   \, \frac{s+1}{2} \, \left(\Var_\nu(f)\right)^{\frac 12} \, .
\end{equation}
Looking at the proof of proposition \ref{thmconcprofile} we see that both situations yield exactly the same bound for the surface measure of a subset of probability $1/2$ i.e the desired result.
\medskip

If $V$ is not smooth we may approximate $\nu$ by convolving with tinny gaussian mollifiers, so that the convolved measures are still log-concave according to Prekopa-Leindler theorem and with smooth potentials. If $X$ has distribution $\nu$ and $G$ is a standard gaussian vector independent of $X$, $\nu_\varepsilon$ will denote the distribution of $X+\varepsilon G$. \\ It is immediate that for a Lipschitz function $f$,
\begin{eqnarray*}
\mathbb E(|f(X+\varepsilon G)-f(X)|) &\leq& \varepsilon \, \parallel |\nabla f|\parallel_\infty \, \mathbb E(|G|) \\ &\leq& \varepsilon \, \sqrt n \, \parallel |\nabla f|\parallel_\infty \, ,
\end{eqnarray*}
so that if $\nu$ satisfies \eqref{eqlogconcosc}, $\nu_\varepsilon$ also satisfies \eqref{eqlogconcosc} with $\beta_\varepsilon(s) = \beta(s) + 2 \varepsilon \, \sqrt n$. We may thus use the result for $\nu_\varepsilon$ and let $\varepsilon$ go to $0$.\\ Assume now that $\nu$ satisfies \eqref{eqlogconcvar}. It holds 
\begin{eqnarray*}
\nu_\varepsilon(|f-\nu_\varepsilon(f)|) &\leq& \nu(|f-\nu(f)|) + 2 \sqrt n \, \varepsilon \, \parallel |\nabla f|\parallel_\infty \\ &\leq& (\beta(s) + 2 \sqrt n \, \varepsilon) \, \parallel |\nabla f|\parallel_\infty \, + \, s \, \left(\Var_\nu(f)\right)^{\frac 12} \, .
\end{eqnarray*}
But, assuming that $\nu(f)=0$ to simplify the notation,
\begin{eqnarray*}
\Var_\nu(f) &=&  \mathbb E(f^2(X+\varepsilon G)) + \mathbb E((f(X+\varepsilon G)-f(X))^2) + 2 \mathbb E(f(X+\varepsilon G)(f(X)-f(X+\varepsilon G))) \\ &\leq& \Var_{\nu_{\varepsilon}}(f) + \left(\mathbb E(f(X+\varepsilon G))\right)^2 + n \, \varepsilon^2 \, \parallel |\nabla f|\parallel^2_\infty + 2 \varepsilon \sqrt n \, \parallel |\nabla f|\parallel_\infty \, \mathbb E(|f(X+\varepsilon G)|)\\&\leq& \Var_{\nu_{\varepsilon}}(f) + 4n \, \varepsilon^2 \, \parallel |\nabla f|\parallel^2_\infty + \, 2 \varepsilon \sqrt n \, \parallel |\nabla f|\parallel_\infty \, \mathbb E(|f(X)|) \, .
\end{eqnarray*}
In particular if $f$ is 1-Lipschitz and bounded by $M$, we get $$\nu_\varepsilon(|f-\nu_\varepsilon(f)|) \leq \left(\beta(s)+2\varepsilon \sqrt n + s(4n\varepsilon^2+ 2M \varepsilon \sqrt n)^{\frac 12}\right) \, + \, s \, \Var_{\nu_{\varepsilon}}(f) $$ i.e. using homogeneity, $\nu_\varepsilon$ also satisfies \eqref{eqlogconcvar} with some $\beta_\varepsilon$, and we may conclude as before. \\ For the median case just remark that $m_{\nu_\varepsilon}(f)$ goes to $m_\nu(f)$ as $\varepsilon$ goes to $0$.
\end{proof}
Using \eqref{eqcheegpoinc} we get similar bounds for $C_P(\nu)$.

\begin{remark}\label{remthmmain}
Of course if $\nu$ satisfies a weak $(1,+\infty)$ Poincar\'e inequality with function $\beta(u)$, we obtain $$C'_C(\nu) \leq \inf_{0\leq s<\frac 12} \, \frac{4 \beta(s)}{\pi \, (\frac 12 -s)^2} \, .$$ Using that $\beta$ is non increasing, it follows that $$C'_C(\nu) \leq \frac{4}{\pi \, (\frac 12 -s_\nu)^4} \quad \textrm{ where } \quad \beta(s_\nu)=\frac{1}{(\frac 12 -s_\nu)^2} \, .$$ We should write similar statements replacing the Oscillation by the Variance. In a sense \eqref{eqlogconcosc} looks more universal since the control quantities in the right hand side do not depend (except the constants of course) of $\nu$. It is thus presumably more robust to perturbations. We shall see this later. Also notice that both \eqref{eqlogconcosc} and \eqref{eqlogconcvar} agree when $s=0$, which corresponds to an explicit bound for the Cheeger constant in E. Milman's theorem. \\ The advantage of \eqref{eqlogconcvar} is that it looks like a deficit in the Cauchy-Schwarz inequality, since we may take $s$ close to 1. \hfill $\diamondsuit$
\end{remark}
Notice that in the non weak framework, a not too far proof (with a functional flavor) is given in \cite{Alonbast} theorem 1.10.
\medskip

\begin{remark}\label{remfortet} Notice that \eqref{eqlogconcosc} is unchanged if we replace $f$ by $f+a$ for any constant $a$, hence we may assume that $\inf f = 0$. Similarly it is unchanged if we multiply $f$ by any $M$, hence we may choose $0\leq f \leq 1$ with $Osc(f)=1$.  \hfill $\diamondsuit$ 
\end{remark}
\medskip

\subsection{Some variations and some immediate consequences.}\label{subsecimmediate}

\subsubsection{Immediate consequences.} \quad Now using our trivial proposition \ref{propwinftyconc} (more precisely the version with the median) we immediately deduce
\begin{corollary}\label{corconc}
For any log-concave probability measure $\nu$, $$C'_C(\nu) \leq \inf_{0<s<\frac 14} \, \frac{16 \, \alpha_\nu^{-1}(s)}{\pi \, (1-4s)^2} \quad \textrm{ and } \quad C_P(\nu) \leq  \inf_{0<s<\frac 14} \, \left(\frac{32 \, \alpha_\nu^{-1}(s)}{\pi \, (1-4s)^2}\right)^2 \, .$$
\end{corollary}
The fact that the concentration profile controls the Poincar\'e or the Cheeger constant of a log-concave probability measure was also discovered by E. Milman in \cite{Emil1}. Another (simpler) proof, based on the semi-group approach was proposed by Ledoux (\cite{led10}). We shall not recall Ledoux's proof, but tuning the constants in this proof furnishes worse constants than ours. The introduction of the weak version of the $(1,+\infty)$ Poincar\'e inequality is what is important here, in order to deduce such a control without any effort.

A similar and even better result was obtained by E. Milman in Theorem 2.1 of \cite{Emil3}, namely $$C'_C(\nu) \leq \, \frac{\alpha_\nu^{-1}(s)}{1-2s}$$ that holds for all $s<\frac 12$. The proof of this result lies on deep geometric results (like the Heintze-Karcher theorem) while ours is elementary. In a sense it is the semi-group approach alternate proof mentioned by E. Milman after the statement of its result.
\medskip

Also notice that the previous corollary gives a new proof of Ledoux's result \eqref{eqled} but with a desperately worse constant. Indeed if we combine the previous bound for $C'_C$ and some explicit estimate in the Gromov-Milman theorem (see respectively \cite{Troy,Beres}) i.e. $$\alpha_\nu(r) \leq 16 \, e^{-r/\sqrt{2C_P}} \quad \textrm { or } \quad \alpha_\nu(r) \leq  \, e^{-r/3 \sqrt{C_P}} \, ,$$ we obtain $(C'_C)^2(\nu) \leq \, m \, C_P(\nu)$ for some $m$. The reader will check that $m$ is much larger than 36.
\medskip

But actually one can recover Ledoux's result in a much more simple way: indeed $$\nu(|f-\nu(f)|) \leq \sqrt{\nu(|f-\nu(f)|^2)} \leq \sqrt{C_P(\nu)} \, \nu^{\frac 12}(|\nabla f|^2) \leq \sqrt{C_P(\nu)} \, \parallel |\nabla f|\parallel_\infty$$ furnishes, thanks to Theorem \ref{thmweakmil}, 
\begin{proposition}\label{propmieuxledoux}
If $\nu$ is a log-concave probability measure, $$C'_C(\nu) \leq C_C(\nu) \leq \frac{16}{\pi} \, \sqrt{C_P(\nu)} \, .$$
\end{proposition}
Since $16/\pi < 6$ this result is slightly better than Ledoux's result recalled in \eqref{eqled}.
\medskip

\begin{remark}\label{remklsmed} Another immediate consequence is the following: since $$|f(x)-f(a)| \leq
\parallel |\nabla f|\parallel_\infty \, |x-a|$$ we have for all $a$, $$\nu(f-m_\nu(f)) = \inf_b \, \int (|f(x)-b|) \nu(dx) \leq \int |f(x)-f(a)| \nu(dx) \leq \parallel|\nabla f|\parallel_\infty \, \int |x-a| \nu(dx) \, .$$ Taking the infimum with respect to $a$ in the right hand side, we thus have 
\begin{equation}\label{eqKLSwidth}
\nu(f-m_\nu(f)) \leq \parallel|\nabla f|\parallel_\infty \, \int |x-m_\nu(x)| \nu(dx) \, \leq \parallel|\nabla f|\parallel_\infty \, \int |x-\nu(x)| \nu(dx) \, .
\end{equation}
A stronger similar result (credited to Kannan, Lov\'{a}sz and Simonovits \cite{KLS}) is mentioned in \cite{Alonbast} p.11, namely 
\begin{equation}\label{eqKLSvar}
\Var_\nu(f) \leq 4 \, \Var_\nu(x) \, \nu(|\nabla f|^2) \, ,
\end{equation}
where $\Var_\nu(x) = \nu (|x -\nu(x)|^2)$.\\
According to \eqref{eqKLSwidth} $$C'_C(\nu) \leq \, \frac{16}{\pi}  \, \int |x-m_\nu(x)| \, d\nu \, \leq \, \frac{16}{\pi} \, \Var^{1/2}_\nu(x) \, .$$ In particular Since $16/\pi < 5,2$, a consequence is the bound $C_P(\nu) \leq  484 \, \Var_\nu(x)$.\\ Notice that this result contains ``diameter'' bounds, i.e. if the support of $\nu$ is compact with diameter $D$, $C'_C(\nu) \leq \frac{16 \, D}{\pi}$. In the isotropic situation one gets $C'_C(\nu) \leq \frac{16 \, \sqrt n}{\pi} \, .$ \hfill $\diamondsuit$
\end{remark}

\begin{remark}\label{remkls2} \quad Consider an isotropic log-concave random vector $X$ with distribution $\nu$. If $f$ is a Lipschitz function we have for all $a$,
\begin{eqnarray}\label{eqsphere}
\nu(|f-a|) &\leq& \mathbb E\left[\left|f(X) - f\left(\sqrt n \, \frac{X}{|X|}\right)\right|\right] + \mathbb E\left[\left|a - f\left(\sqrt n \, \frac{X}{|X|}\right)\right|\right] \nonumber \\ &\leq& \parallel |\nabla f| \parallel_\infty \, \mathbb E\left[\left| |X| - \sqrt n\right|\right] + \mathbb E\left[\left|a - f\left(\sqrt n \, \frac{X}{|X|}\right)\right|\right] \, .
\end{eqnarray}
Hence, if we choose $a=\mathbb E\left[f\left(\sqrt n \, \frac{X}{|X|}\right)\right]$ and if we denote by $\nu_{angle}$ the distribution of $X/|X|$ we obtain 
\begin{equation}\label{eqsphere2}
\nu(|f-m_\nu(f)|) \leq 2 \, \parallel |\nabla f| \parallel_\infty \, \left(\mathbb E\left[\left| |X| - \sqrt n\right|\right] + \sqrt n \, \sqrt{C_P(\nu_{angle})}\right).
\end{equation}
This shows that the Cheeger constant of $\nu$ is completely determined by the concentration of the radial part of $X$ around $\sqrt n$ (which is close to its mean), and the Poincar\'e constant (we should also use the Cheeger constant) of $X/|X|$. 

In particular, if $\nu$ is spherically symmetric, the distribution of $X/|X|$ is the uniform measure on the sphere $S^{n-1}$ which is known to satisfy (provided $n\geq 2$) a Poincar\'e inequality with Poincar\'e constant equal to $1/n$ for the usual euclidean gradient (not the riemanian gradient on the sphere). In addition, in this situation its known that $$\mathbb E\left[\left| |X| - \sqrt n\right|\right] \leq 1 \quad \quad \textrm{ see \cite{bobsphere} formula (6)}$$ so that we get that $C'_C(\nu) \leq \frac{64}{\pi}$ for an isotropic radial log-concave probability measure. Since $\frac{16}{\pi} \sqrt {12} < \frac{64}{\pi}$, this result is worse than the one proved in \cite{bobsphere} telling that $C_P(\nu) \leq 12$ so that $C'_C(\nu) \leq \frac{16}{\pi} \sqrt {12}$ thanks to proposition \ref{propmieuxledoux} ($12$ may replace the original $13$ thanks to a remark by N. Huet \cite{Nolwen}). Actually, applying \eqref{eqKLSvar} in dimension $1$, it seems that we may replace $12$ by $4$. \hfill $\diamondsuit$
\end{remark}
\medskip

\subsubsection{Variations: the $\mathbb L^2$ framework.} \quad In some situations it is easier to deal with variances.  We recalled that any (nice) absolutely continuous probability measure satisfies a weak $(2,2)$ Poincar\'e inequality. 
\begin{theorem}\label{thmweakmil2}
Let $\nu$ a log-concave probability measure satisfying the weak $(2,2)$ Poincar\'e inequality, for some $s<\frac 16$, $$\Var_\nu(f) \leq \beta(s) \nu(|\nabla f|^2) + s \, Osc^2(f) \, .$$ Then $$C'_C(\nu) \leq \frac{4 \, \sqrt{\beta(s) \ln 2}}{1-6s} \quad \textrm{ and} \quad C_P(\nu) \leq 4 \, (C'_C(\nu))^2 \, .$$
\end{theorem}
\begin{proof}
We start with the following which is also due to Ledoux: if $\nu$ is log-concave, then for any subset $A$, $$\sqrt{t} \, \nu_s(\partial A) \, \geq \, \nu(A) - \, \nu\left((P_t \mathbf 1_A)^2\right) \, .$$ But $$\nu\left((P_t \mathbf 1_A)^2\right) = \Var_\nu(P_t \mathbf 1_A) + \left(\nu(P_t \mathbf 1_A)\right)^2= \Var_\nu(P_t \mathbf 1_A) + \nu^2(A) \, .$$ Define $u(t)=\Var_\nu(P_t \mathbf 1_A)$. Using the semi-group property and the weak Poincar\'e inequality it holds
$$
\frac{d}{dt} u(t) = -2 \, \nu(|\nabla P_t \mathbf 1_A|^2) \leq \frac{-2}{\beta(s)} \, (u(t) - s)
$$
since $Osc(P_t \mathbf 1_A)\leq 1$. Using Gronwall's lemma we thus obtain
$$\Var_\nu(P_t \mathbf 1_A) \leq e^{-2t/\beta(s)} \, \nu(A) + s \, \left(1-e^{-2t/\beta(s)}\right)$$ so that finally, if $\nu(A)=1/2$ we get $$\sqrt{t} \, \nu_s(\partial A) \geq \left(1-e^{-2t/\beta(s)}\right) \, \left(\frac 12 - s\right) \, - \, \frac 14 \, .$$ Choose $t= \beta(s) \, \ln 2$. The right hand side in the previous inequality becomes $\frac 18 (1-6s)$, hence $$Is_\nu(1/2) \geq \frac{1-6s}{8 \, \sqrt{\beta(s) \, \ln 2}} \, .$$ Hence the result arguing as in the previous proof.
\end{proof}
\medskip

\subsubsection{Other consequences. Reducing the support.}\label{subsubsecreduc} \quad
All the previous consequences are using either the weak or the strong $(1,+\infty)$ Poincar\'e inequality. The next consequence will use the full strength of what precedes.\\ Pick some Borel subset $A$. Let $a \in \mathbb R$ and $f$ be a smooth function. Then
$$
\nu(|f-a|) \leq \int_A \, |f-a|d\nu + (1-\nu(A)) \, \parallel f-a\parallel_\infty \leq \int_A \, |f-a|d\nu + (1-\nu(A)) \, Osc(f) \, .
$$
Denote  $d\nu_A = \frac{\mathbf 1_A}{\nu(A)} \, d\nu$ the restriction of $\nu$ to $A$. Choosing $a=m_{\nu_A}(f)$ we have 
\begin{eqnarray*}
\nu(|f-m_\nu(f)|) &\leq& \nu(|f-a|) \leq \nu(A) \, \nu_A(|f-a|) + (1-\nu(A)) \, Osc(f) \\ &\leq& \nu(A) \, \left(\beta_{\nu_A}(u) \, \parallel |\nabla f|\parallel_\infty + u \, Osc(f)\right) \, + \, (1-\nu(A)) \, Osc(f)\\ &\leq& 
\nu(A) \, \beta_{\nu_A}(u) \, \parallel |\nabla f|\parallel_\infty + (1-\nu(A)(1-u)) \, Osc(f) \, ,
\end{eqnarray*}
provided $\nu_A$ satisfies some $(1,+\infty)$ Poincar\'e inequality. We can improve the previous bound, if $\nu_A$ satisfies some Cheeger inequality and get $$\nu(|f-m_\nu(f)|) \leq \, \nu(A) \, C'_C(\nu_A) \, \nu(|\nabla f|) + (1-\nu(A)) \, Osc(f) \, .$$
Hence, applying theorem \ref{thmweakmil} or proposition \ref{thmconcprofile} we have
\begin{proposition}\label{proprestrict}
Let $\nu$ be a log-concave probability measure, $A$ be any subset with $\nu(A)> \frac 12$ and $d\nu_A = \frac{\mathbf 1_A}{\nu(A)} \, d\nu$ be the (normalized) restriction of $\nu$ to $A$. Then 
\begin{enumerate}
\item $$C'_C(\nu) \leq \frac{\nu(A) \, C'_C(\nu_A)}{2\nu(A)-1} \, ,$$ 
\item if $\nu_A$ satisfies a $(1,+\infty)$ weak Poincar\'e inequality with rate $\beta_{\nu_A}$, then as soon as $u<1 - \frac{1}{2\nu(A)}$, $$\beta_\nu(u) \leq \nu(A) \, \beta_{\nu_A}\left(1 - \frac{1-u}{\nu(A)}\right)$$ so that $$C'_C(\nu) \leq \frac{4 \nu(A) \, \beta_{\nu_A}(u)}{\pi \, ((1-u)\nu(A)-\frac 12)^2} \, .$$
\end{enumerate}
\end{proposition}
\begin{remark}\label{remsuppemil}
A similar result is contained in \cite{Emil1} namely if $K$ is some convex body, $$C'_C(\nu) \leq \frac{1}{\nu^2(K)} \, C'_C(\nu_K) \, .$$ This result is similar when $\nu(K)$ is close to 1, but it requires the convexity of $K$. Convexity of $K$ ensures that $\nu_K$ is still log-concave. We shall come back to this point later. Of course our result does not cover the situation of sets with small measure. \hfill $\diamondsuit$
\end{remark}
\medskip

This result enables us to reduce the study of log-concave probability measures to the study of compactly supported distributions, arguing as follows. Let $Z$ be a random variable (in $\mathbb R^n$) with log concave distribution $\nu$. We may assume without loss of generality that $Z$ is centered. Denote by $\sigma^2(Z)$ the largest eigenvalue of the covariance matrix of $Z$.
\begin{itemize}
\item \quad \textit{$l^2$ truncation. \\ \\} 
Thanks to Cebicev inequality, for $a>1$, $$\mathbb P(|Z|>a \, \sigma(Z) \, \sqrt n) \leq \frac {1}{a^2} \, .$$ According to Proposition \ref{proprestrict}, if $a>\sqrt 2$, $$C'_C(Z) \leq \frac{a^2}{a^2-2} \, C'_C(Z(a))$$ where $Z(a)$ is the random variable $\mathbf 1_{Z\in K_a} \, Z$ supported by  $K_a=B(0,a \, \sigma(Z) \, \sqrt n)$ with distribution $\frac{\mathbf 1_{K_a}}{\nu(K_a)} \, \nu$.  Of course the new variable $Z(a)$ is not necessarily centered, but we may, without changing the Poincar\'e constant(s),  consider the variable $\bar Z(a) =Z(a) -\mathbb E(Z(a))$. It is easily seen that  $$|\mathbb E(Z_i(a))|\leq  |\mathbb E(-Z_i \, \mathbf 1_{K_a^c}(Z))|\leq \frac{\Var^{1/2}(Z_i)}{a} \leq \frac{\sigma(Z)}{a} \, ,$$ i.e. $$\sum_{i=1}^n \, |\mathbb E(Z_i(a))|^2 \leq \frac{n\sigma^2(Z)}{a^2}$$ so that $\bar Z(a)$ is centered and supported by $B(0,\sqrt{a^2+(1/a^2)} \, \sigma(Z) \, \sqrt n)$. \\ Notice that for all $i$, 
\begin{equation}\label{eqvartrans}
\Var(Z_i) \geq \Var(\bar Z_i(a)) \, \geq \, \Var(Z_i) \, \left(1-\frac{\kappa}{a} - \frac{1}{a^2}\right) \, ,
\end{equation}
 where $\kappa$ is the universal Khinchine constant, i.e. satisfies $$\mu(z^4) \leq \kappa^2 \, (\mu(z^2))^2$$ for all log concave probability measure on $\mathbb R$. According to the one dimensional estimate of Bobkov (\cite{bob99} corollary 4.3) we already recalled, we know that $\kappa \leq 7$. \\The upper bound in \eqref{eqvartrans} is immediate, while for the lower bound we use the following reasoning :
\begin{eqnarray*}
\Var(\bar Z_i(a)) &=& \mathbb E(Z_i^2 \, \mathbf 1_{K_a}(Z)) - (\mathbb E(Z_i(a)))^2 \\ &\geq&  \Var(Z_i) -  \, \mathbb E(Z_i^2 \, \mathbf 1_{K_a^c}(Z)) - (\mathbb E(Z_i(a)))^2\\ &\geq&  \Var(Z_i) \, - \, \frac 1a \, (\mathbb E(Z_i^4))^\frac 12 \, - \frac{\Var(Z_i)}{a^2}  
\end{eqnarray*}
according to the previous bound on the expectation. We conclude by using Khinchine inequality.\\ Similar bounds are thus available for $\sigma^2{\bar Z(a)}$ in terms of $\sigma^2(Z)$.

\begin{remark}\label{best constant} \quad Though we gave explicit forms for all the constants, they are obviously not sharp. \\ For instance we used the very poor Cebicev inequality for reducing the support of $\nu$ to some euclidean ball, while much more precise concentration estimates are known. For an isotropic log-concave random variable (vector) $Z$, it is known that its distribution satisfies some ``concentration'' property around the sphere of radius $\sqrt n$. We shall here recall the best known result, due to Gu\'edon and E. Milman (\cite{GEM}) and we refer to \cite{Gia} chapter 13 for a complete overview of the state of the art:

\begin{theorem}\label{thmgm}[Gu\'edon and E. Milman]
Let $Z$ be an isotropic log-concave random vector in $\mathbb R^n$. Then there exist some universal positive constants $C$ and $c$ such that for all $t>0$, $$\mathbb P(| \, |Z|-\sqrt n \, | \geq t \, \sqrt n) \,  \leq \, C \, \exp ( - \, c \sqrt n \, \min\{t,t^3\}) \, . \quad \quad \quad \diamondsuit $$
\end{theorem}
\end{remark}
\medskip

\item
{$l^\infty$ truncation. \\ \\} \quad Instead of looking at euclidean balls we shall look at hypercubes, i.e. $l^\infty$ balls. We assume that $Z$ is centered.
\medskip

\noindent According to Prekopa -Leindler theorem again, we know that the distribution $\nu_1$ of $Z_1$ is a log-concave distribution with variance $\lambda^2_1= \Var(Z_1)$. Hence $\nu_1$ satisfies a Poincar\'e inequality with $C_P(\nu_1) \leq 12 \, \lambda^2_1$ according to \cite{bob99} proposition 4.1. Of course a worse bound was obtained in remark \ref{remkls2}. \\ Using Proposition 4.1 in \cite{BL97} (see lemma \ref{lemBL97} in the next section), we have for all $1\ge\epsilon>0$
\begin{equation}\label{eqconcentrat}
\nu_1(z_1>a_1) \leq \frac{4-\epsilon}{\epsilon} \, e^{- \, \frac{(2-\epsilon)a_1}{\sqrt{C_P(\nu_1)}}} \leq \frac{4-\epsilon}{\epsilon} \, e^{- \, \frac{(2-\epsilon)a_1}{2 \sqrt 3 \, \lambda_1}} \, .
\end{equation} 
\noindent Of course changing $Z$ in $-Z$ we get a similar bound for $\nu_1(z_1<- a_1)$. Choosing $a_1=\frac{2\sqrt 3}{2-\epsilon} \, \lambda_1 \, a \, \ln n$ for some $a>0$, we get  $$\nu_1(|z_1|>a_1) \leq 2\frac{4-\epsilon}{\epsilon n^a} \, .$$ Hence if 
\begin{equation}\label{eqhyper}
K_a=\left\{\max_{i=1,...,n} |z_i|< \frac{2 \sqrt 3}{2-\epsilon} \, \sigma(Z) \, a \, \ln n\right\}
\end{equation}
\noindent we have $$\nu(K_a) \geq 1 \, - \, 2\frac{4-\epsilon}{\epsilon n^{a-1}} \, .$$ 
\medskip

\noindent Using Proposition \ref{proprestrict} we have 
\begin{proposition}\label{prophyper}
For $n\geq 2$, $K_a$ being defined by \eqref{eqhyper}, we have $$C'_C(\nu)\leq \frac{n^{a-1}}{n^{a-1}-(8-2\epsilon)\epsilon^{-1}} \; C'_C(\nu_{K_a}) \, .$$ 
\end{proposition}
\noindent Again we may center $\nu_{K_a}$, introducing a random vector $\bar Z$ with distribution equal to the re-centered $\nu_{K_a}$. This time it is easily seen that $$(1-\varepsilon(n)) \, \Var(Z_i) \leq \Var(\bar Z_i)\leq \, \Var(Z_i) \, ,$$ with $\varepsilon(n) \to 0$ as $n \to +\infty$.
\medskip

\noindent Notice that we have written a more precise version of Latala's deviation result (\cite{Latala})
\begin{theorem}\label{thmlatal}
Let $Z$ be an isotropic log-concave random vector in $\mathbb R^n$. Then for $n\geq 2$, $1\ge \epsilon>0$, $$\mathbb P\left(\max_{i=1,..., n} \, |Z_i| \geq t \ln n\right) \leq \, \frac{(8-2\epsilon)\epsilon^{-1}}{n^{\frac{(2-\epsilon)t}{2\sqrt 3}-1}} \quad \textrm{ for } t\geq \frac{2\sqrt 3}{2-\epsilon} \,  .$$
\end{theorem}
\end{itemize}
\smallskip

Let us give another direct application.  $\nu$ is said to be unconditional if it is invariant under the transformation $(x_1,...,x_n) \mapsto (\varepsilon_1 \, x_1,...,\varepsilon_n \, x_n)$ for any $n$-uple $(\varepsilon_1,...,\varepsilon_n) \in \{-1,+1\}^n$. Defining $K_a$ as in \eqref{eqhyper} (with $\sigma(Z)=1$ here), the restricted measure $\nu_{K_a}$ is still unconditional.

So we may apply Theorem 1.1 in \cite{CorGoz} (and its proof in order to find an explicit expression of the constant) saying that $\nu_{K_a}$ satisfies some weighted Poincar\'e inequality $$\Var_{\nu_{K_a}}(f) \, \leq \, (4\sqrt 3 +1)^2 \, \sum_{i=1}^n \, \nu_{K_a}\left(\nu_{K_a}^{i-1}(x_i^2) \, (\partial_i f)^2\right) \, ,$$ where $\nu_{K_a}^{j}$ denotes the conditional distribution of $\nu_{K_a}$ w.r.t. the sigma field generated by $(x_1,...,x_j)$. Actually, up to the pre-constant, we shall replace the weight $\nu_{K_a}^{i-1}(x_i^2)$ by the simpler one $x_i^2+\nu_{K_a}(x_i^2)$ according to \cite{Klartoul}. 

In all cases, since $x_i$ is $\nu_{K_a}$ almost surely bounded by a constant times $\ln(n)$, we have obtained thanks to Proposition \ref{prophyper} the following result originally due to Klartag \cite{Klartuncond}
\begin{proposition}\label{propuncond}
There exists an universal constant $c$ such that, if $\nu$ is an isotropic and unconditional log-concave probability measure, $$C_P(\nu) \leq c \, \max(1,\ln^2 n) \, .$$ 
\end{proposition}

Of course what is needed here is the invariance of $\nu_{K_a}$ with respect to some symmetries (see \cite{CorGoz,BarCor}). But it is not easy to see how $\nu_{K_a}$ inherits such a property satisfied by the original $\nu$, except in the unconditional case.
\bigskip

\section {Transference of the Poincar\'e inequality.}\label{sectransfer}   \medskip

\subsection{Transference via absolute continuity. \\ \\}\label{subsectrans1}
If $\mu$ and $\nu$ are probability measures it is well known that $$C_P(\nu) \; \leq \; \parallel\frac{d\nu}{d\mu}\parallel_\infty \; \parallel \frac{d\mu}{d\nu}\parallel_\infty \; C_P(\mu) \, ,$$ the same being true with the same pre-factor if we replace $C_P$ by $C_C$ or $C'_C$. Similar results are known for weak $(2,2)$ Poincar\'e inequalities too. In this section we shall give several transference principles allowing us to reinforce this result, at least under some curvature assumption.  Some transference results have been obtained in Theorem 1.5 in \cite{Bemil} using transference results for the concentration profile. We shall come back later to this point. 
\medskip

\begin{theorem}\label{thmtransf1}
Let $\nu$ and $\mu$ be two probability measures.  
\begin{enumerate}
\item  \quad If for some $1<p$, $$\int \, \left|\frac{d\nu}{d\mu}\right|^p \, d\mu \, = \, M_p^p \, < \, + \infty \, ,$$ then $$\beta_\nu(s) \, \leq \, M^{p/(p-1)}_p \, s^{-1/(p-1)} \, C(\mu) \, ,$$ where $C(\mu)$ can be chosen as $C'_C(\mu)$, $1/B_{1,\infty}(\mu)$, $C_C(\mu)$ or $\sqrt{C_P(\mu)}$. \\ In particular if $\nu$ is log-concave
$$C'_C(\nu) \, \leq \, D \, C(\mu) \, M^{p/(p-1)}_p$$ where $D=\frac{16 \,  (p+1)^{1/(p-1)}}{\pi \, (\frac 12 - \frac{1}{p+1})^2 }$.\\
\item \quad Let $\nu$ be log-concave. If the relative entropy $D(\nu||\mu):= \int \, \ln(d\nu/d\mu) \, d\nu$ is finite, then with the same $C(\mu)$ and any $u<\frac 12$, $$C'_C(\nu) \, \leq \, \frac{4 \, \left(e^{2 \max(1,D(\nu||\mu))/u} - 1\right)}{\pi \, (\frac 12 - u)^2} \, C(\mu) \, .$$
\end{enumerate}
\end{theorem}
\begin{proof}
Let $f$ be a smooth function. It holds
\begin{eqnarray}\label{eqdecomp}
\nu(|f-m_\nu(f)|) &\leq& \nu(|f-a|) \nonumber \\ &\leq& 
\int \, |f-a| \, \frac{d\nu}{d\mu} \, d\mu \, \nonumber \\ &\leq& K \, \int \, \mathbf 1_{d\nu/d\mu\leq K} \, |f-a| \, d\mu \, + \, Osc(f) \, \int \, \mathbf 1_{d\nu/d\mu>K} \, d\nu  \nonumber \\ &\leq& K \, \int \,  |f-a| \, d\mu \, + \, Osc(f) \, \int \, \mathbf 1_{d\nu/d\mu>K} \, d\nu  \, .\end{eqnarray}
In order to control the last term we may use H\"{o}lder (or Orlicz-H\"{o}lder) inequality. In the usual $\mathbb L^p$ case we have, using Markov inequality 
\begin{eqnarray*}
\int \, \mathbf 1_{d\nu/d\mu>K} \, d\nu &\leq& M_p \, \mu^{\frac 1q}\left(\frac{d\nu}{d\mu}>K\right) \\ &\leq& M_p \, \frac{M_p^{p/q}}{K^{p/q}} \, = \, \frac{M_p^{1+(p/q)}}{K^{p/q}} \, = \, \frac{M_p^p}{K^{p-1}} \, ,
\end{eqnarray*}
so that choosing $K=M_p^{p/(p-1)} \, u^{-1/(p-1)}$ we have obtained $$\nu(|f-m_\nu(f)|) \leq M_p^{p/(p-1)} u^{-1/(p-1)} \, \int \, |f-a| \, d\mu \, + \, u \, Osc(f) \, .$$ Then using either $a=m_\mu(f)$ or $a=\mu(f)$ we get the desired result, with $C(\mu)=C'_C(\mu)$ or $1/B_{1,\infty}(\mu)$ or $C_C(\mu)$ or $\sqrt{C_P(\mu)}$. If in addition $\nu$ is log-concave it follows $$C'_C(\nu) \, \leq \, \frac{16}{\pi \, (\frac 12 - u)^2 \, u^{1/(p-1)}} \, M_p^{\frac{p}{p-1}} \, C(\mu) \, .$$  It remains to optimize in $u$ for $0<u<\frac 12$ and elementary calculations show that the maximum is attained for $u=1/(p+1)$.
\smallskip

Starting with \eqref{eqdecomp} we may replace H\"{o}lder's inequality by Orlicz-H\"{o}lder's inequality for a pair of conjugate Young functions $\theta$ and $\theta^*$, that is $$\int \, \mathbf 1_{d\nu/d\mu>K} \, d\nu \leq \, 2 \, \parallel \frac{d\nu}{d\mu} \parallel_{\mathbb L_\theta(d\mu)} \, \parallel \mathbf 1_{d\nu/d\mu>K}  \parallel_{\mathbb L_{\theta^*(d\mu)}} \, .$$ Here the chosen Orlicz norm is the usual gauge (Luxemburg) norm, i.e. $$\parallel h \parallel_{L_\theta(\mu)} = \inf \{b\geq 0 \quad s.t. \quad \int \theta(|h|/b) \, d\mu \leq 1 \} \, ,$$ and recall that for any $\lambda>0$, 
\begin{equation}\label{eqorl1}
\parallel h \parallel_{L_\theta(\mu)} \leq \frac 1\lambda \, \max(1 \, , \, \int \, \theta(\lambda |h|) d\mu) \, .
\end{equation}
For simplicity we will perform the calculations only for the pair of conjugate Young functions $$\theta^*(u)=e^{|u|}-1 \quad , \quad \theta(u)=(|u| \ln(|u|) +1 - |u|) \, .$$ According to what precedes $$\parallel \frac{d\nu}{d\mu} \parallel_{\mathbb L_\theta(d\mu)} \, \leq \, \max(1 \, , \, D(\nu||\mu)) \, $$ and 
\begin{eqnarray*}
\int \, \theta(\mathbf 1_{d\nu/d\mu>K}/b) \, d\mu \, &=& \, (e^{1/b}-1) \, \mu\left(\frac{d\nu}{d\mu}>K\right) \\ &\leq& (e^{1/b}-1) \, \frac 1K \, ,
\end{eqnarray*}
the final bound being non optimal since we only use $\mathbf 1_{d\nu/d\mu>K} \leq \frac 1K \, \frac{d\nu}{d\mu}$ and not the better integrability of the density. Using the best integrability does not substantially improve the bound. We thus obtain $$\parallel \mathbf 1_{d\nu/d\mu>K}  \parallel_{\mathbb L_{\theta^*(d\mu)}} = \frac{1}{\ln (1+K)}$$ and we may conclude as before.
\end{proof}
The method we used in the previous proof is quite rough. One can expect (in particular in the entropic case) to improve upon the constants, using more sophisticated tools. This is the goal of the next result
\begin{theorem}\label{thmtransf1bis}
Let $\nu$ and $\mu$ be two probability measures.  
\begin{enumerate}
\item  \quad If for some $1<p\leq 2$, $$\int \, \left|\frac{d\nu}{d\mu}\right|^p \, d\mu \, = \, M_p^p \, < \, + \infty \, ,$$ then $$\frac{1}{B_{1,\infty}(\nu)} \, \leq \, \left(\frac{p}{p-1} \, 8^{\frac{p}{(p-1)}} \, C_P(\mu)\right)^{\frac 12} \, M_p \, .$$ If in addition $\nu$ is log-concave $$C'_C(\nu) \, \leq \, \frac{16 \sqrt p}{\pi \, \sqrt{p-1}} \, 8^{\frac{p}{2(p-1)}} \, \sqrt{C_P(\mu)} \, M_p \, .$$ \\
\item \quad If the relative entropy $D(\nu||\mu):= \int \, \ln(d\nu/d\mu) \, d\nu$ is finite, then $$\frac{1}{B_{1,\infty}(\nu)} \, \leq \, 2 \, \sqrt{C_P(\mu)} \, \max\left(1 \, , \, 3 e^{\sqrt{C_P(\mu)}}\right) \, \max\left(1 \, , \, D(\nu||\mu)\right) \, .$$ If in addition $\nu$ is log-concave
$$ C'_C(\nu) \, \leq \, \frac{32}{\pi} \, \sqrt{C_P(\mu)} \, \max\left(1 \, , \, 3 e^{\sqrt{C_P(\mu)}}\right) \, \max\left(1 \, , \, D(\nu||\mu)\right) \, .$$
\end{enumerate}
\end{theorem}
\begin{proof}
Let $f$ be a smooth function. We have 
\begin{eqnarray}\label{eqhold}
\frac 12 \, \nu(|f-\nu(f)|) &\leq& \nu(|f-m_\nu(f)|) \leq \nu(|f-\mu(f)|) \nonumber \\
&\leq& \left(\int \, |f-\mu(f)|^q \, d\mu\right)^{1/q} \, \, \left(\int  \, \left|\frac{d\nu}{d\mu}\right|^p \, d\mu\right)^{1/p} \, .
\end{eqnarray}
Now we can use the results in \cite{CGR}, in particular the proof of Theorem 1.5 (see formulae 2.3, 2.7 and 2.8 therein) where the following is proved
\begin{lemma}\label{lemcgr}
For all $q\geq 2$, it holds $$\int \, |f-\mu(f)|^q \, d\mu \, \leq \, q \, 8^q \, C_P(\mu) \, \int \, |f-\mu(f)|^{q-2} \, |\nabla f|^2 \, d\mu \, .$$
\end{lemma}
It is at this point that we need $p\leq 2$. Using H\"{o}lder inequality we deduce
\begin{equation}\label{eqcgr}
\int \, |f-\mu(f)|^q \, d\mu \leq  \, (q \, 8^q \, C_P(\mu))^{\frac q2}  \, \int |\nabla f|^q \, d\mu  \, .
\end{equation}
It follows $$\nu(|f-m_\nu(f)|) \leq \, (q \, 8^q \, C_P(\mu))^{\frac 12} \, M_p \, \parallel |\nabla f| \parallel_\infty \, .$$ Since $\nu$ is log-concave, we get the desired result, using theorem \ref{thmweakmil}.
\smallskip

\noindent Now we turn to the second part of the Theorem which is based on Proposition 4.1 in \cite{BL97} we recall now
\begin{lemma}\label{lemBL97}[Bobkov-Ledoux]
If $g$ is Lipschitz, $$\mu(e^{\lambda g}) \,  \leq \, \left(\frac{2 + \lambda \, C_P^{\frac 12}(\mu) \, \parallel |\nabla g|\parallel_\infty}{2 - \lambda \, C_P^{\frac 12}(\mu) \, \parallel |\nabla g|\parallel_\infty}\right) \, \, e^{\lambda \mu(g)} \, ,$$ provided $2 > \lambda \, C_P^{\frac 12}(\mu) \, \parallel |\nabla g|\parallel_\infty > \, 0$.
\end{lemma}
Hence, in \eqref{eqhold}, we may replace the use of H\"{o}lder inequality by the one of the Orlicz-H\"{o}lder inequality 
$$\nu(|f-m_\nu(f)|) \leq 2 \, \parallel f-\mu(f)\parallel_{L_\theta(\mu)} \, \parallel \frac{d\nu}{d\mu}\parallel_{L_{\theta^*}(\mu)} \, .$$ 
Again we are using the pair of conjugate Young functions $$\theta(u)=e^{|u|}-1 \quad , \quad \theta^*(u)=|u| \ln(|u|) +1 - |u| \, .$$   Without lack of generality we can first assume that $f$ (hence $f-\mu(f)$) is 1-Lipschitz. \\ We then apply \eqref{eqorl1} and Lemma \ref{lemBL97} with $g=|f-\mu(f)|$ and $\lambda= 1/\sqrt{C_P(\mu)}$. Since 
$$\mu(|g|)\leq \mu^{\frac 12}(|g|^2) \leq C^{\frac 12}_P(\mu) \, \mu^{\frac 12}(|\nabla g|^2) \leq C^{\frac 12}_P(\mu) \, \parallel |\nabla g|\parallel_\infty \, ,$$ we obtain $$\parallel f-\mu(f)\parallel_{L_\theta(\mu)} \leq \sqrt{C_P(\mu)} \, \max\left(1 \, , \, 3 e^{\sqrt{C_P(\mu)}}\right) \, \parallel |\nabla f|\parallel_\infty \, .$$ Similarly $$\parallel \frac{d\nu}{d\mu}\parallel_{L_{\theta^*}(\mu)} \leq \max\left(1 \, , \, D(\nu||\mu)\right) \, .$$ Again we conclude thanks to  theorem \ref{thmweakmil}.
\end{proof}
As usual we should try to optimize in $p$ for $p$ going to $1$ depending on the rate of convergence of $M_p$ to $1$. \\ \noindent We cannot really compare both theorems, since the constant $C(\mu)$ in the first theorem can take various values, while it is the usual Poincar\'e constant in the second theorem. In the $\mathbb L^p$ case, the first theorem seems to be better for large $p's$ and the second one for small $p's$. In the entropic case, the second one looks better. 

\noindent Finally, let us recall the following beautiful transference result between two log-concave probability measures proved in \cite{Emil1} and which is a partial converse of the previous results
\begin{proposition}\label{propreduc}[E. Milman]
Let $\mu$ and $\nu$ be two log-concave probability measures.  Then $$C'_C(\nu) \leq \Big |\Big| \frac{d\mu}{d\nu}\Big |\Big|_\infty^2 \, C'_C(\mu)  \, .$$
\end{proposition}
\bigskip

\subsection{Transference using concentration. \\ \\}\label{transfconc}
These results have to be compared with the ones one can obtain using the concentration functions. We first recall the statement of Proposition 2.2 in \cite{Bemil}, in a simplified form :
\begin{proposition}\label{propbemil}[Barthe-Milman]
Denote by $\alpha_\nu$ the concentration profile of a probability measure $\nu$, i.e. $$\alpha_\nu(r) := \sup \{1 - \, \nu(A+B(0,r)) \, ; \, \nu(A) \geq \frac 12\} \, , \, r>0 \, ,$$ where $B(y,r)$ denotes the euclidean ball centered at $y$ with radius $r$.

Assume that for some $1<p\leq + \infty$, $$\int \, \left|\frac{d\nu}{d\mu}\right|^p \, d\mu \, = \, M_p^p \, < \, + \infty \, .$$ Then if $q=\frac{p}{p-1}$, for all $r>0$, $$\alpha_\nu(r) \leq 2M_p \, \alpha^{1/q}_\mu(r/2) \, .$$ 
\end{proposition}
We may use this result together with corollary \ref{corconc} to deduce
\begin{corollary}\label{corconc1}
Under the assumptions of Proposition \ref{propbemil}, if $\nu$ is log-concave, $$C'_C(\nu) \leq \, \inf_{0<s<\frac 14} \, \frac{32 \, \alpha_\mu^{-1}((s/2M_p)^q)}{\pi(1-4s)^2} \, .$$
\end{corollary}
\medskip

\subsection{Transference using distances. \\ \\}\label{subsecW}
As shown in \cite{Emil1} theorem 5.5, the ratio of the Cheeger constants of two log-concave probability measures is controlled by their total variation distance (which is the half of the $W_0$ Wasserstein distance). Recall the equivalent definitions of the total variation distance
\begin{definition}\label{deftv}
If $\mu$ and $\nu$ are probability measures the total variation distance $d_{TV}(\mu,\nu)$ is defined by one of the following equivalent expressions
\begin{eqnarray*}
d_{TV}(\mu,\nu) &:=& \quad \frac 12 \, \sup_{\parallel f\parallel_\infty \leq 1} \, \left|\mu(f)-\nu(f)\right| \\ &=& \quad  \sup_{0 \leq f \leq 1} \, \left|\mu(f)-\nu(f)\right| \\ &=& \quad \inf\{\mathbb P(X\neq Y) \; ; \; \mathcal L(X)=\mu \, , \, \mathcal L(Y)=\nu \, \}
\end{eqnarray*}
which is still equal to $$\frac 12 \, \int \, \left| \frac{d\mu}{dx} \, - \, \frac{d\nu}{dx}\right| \, dx $$ when $\mu$ and $\nu$ are absolutely continuous w.r.t. Lebesgue measure. \\ The second equality is immediate just noticing that for $0\leq f\leq 1$, $\mu(f-\frac 12)-\nu(f-\frac 12)=\mu(f)-\nu(f)$ and that $\parallel f-\frac 12\parallel_\infty \leq \frac 12$. 
\end{definition}
More precisely one can show the following explicit result
\begin{theorem}\label{corequiv}
Let $\mu(dx)=e^{-V(x)}dx$ and $\nu(x)=e^{-W(x)}dx$ be two log-concave probability measures. If $d_{TV}(\mu,\nu) = \, 1 - \varepsilon$, for some $\varepsilon >0$, then $$C'_C(\nu) \, \leq \, \frac{\kappa}{\varepsilon^2} \, \left(1 \, \vee \, \ln \, \left(\frac{1}{\varepsilon}\right)\right) \, C'_C(\mu) \, ,$$ for some universal constant $\kappa$ one can choose equal to $(192 \, e/\pi)$. 
\end{theorem}
\begin{proof} We give a short proof (adapted from \cite{Emil1}) that does not use concentration results. \\ First if $Z_\mu$ and $Z_\nu$ are random variables with respective distribution $\mu$ and $\nu$, and $\lambda>0$ the total variation distance between the distributions of $\lambda Z_\mu$ and $\lambda Z_\nu$ is unchanged, hence still equal to $1-\varepsilon$. Choosing $\lambda = 1/\sqrt{C_P(\mu)}$ we may thus assume that $C_P(\mu)=1$. \\ Introduce the probability measure $\theta(dx)= \frac{1}{\varepsilon} \, \min(e^{-V(x)},e^{-W(x)}) \, dx$ which is still log-concave and such that $d\theta/d\mu$ and $d\theta/d\nu$ are bounded by $1/\varepsilon$.  Using proposition \ref{propreduc} we first have $C'_C(\nu) \leq \frac{1}{\varepsilon^2} \, C'_C(\theta)$. Next $D(\theta||\mu) \leq \ln(1/\varepsilon)$ so that using theorem \ref{thmtransf1bis} (2) we have $C'_C(\theta) \leq \frac{96 \, e}{\pi} \max(1,\ln(1/\varepsilon))$. It follows that $$C'_C(\nu) \leq \frac{96 \, e}{\pi \, \varepsilon^2} \max(1,\ln(1/\varepsilon))$$ provided $C_P(\mu)=1$. It remains to use $C'_C(\lambda Z_\nu)=\lambda \, C'_C(\nu)= (1/\sqrt{C_P(\mu)}) \, C'_C(\nu)$ and  $\sqrt{C_P(\mu)} \leq 2 C'_C(\mu)$ to get the result.
\end{proof}
\medskip

But since the $(1,+\infty)$ Poincar\'e inequality deals with Lipschitz functions, it is presumably more natural to consider the $W_1$ Wasserstein distance $$W_1(\nu,\mu):= \sup_{f \, 1-Lipschitz} \, \int \, f \, (d\mu-d\nu) \, = \, \inf_{\mathcal L(X)=\mu,\mathcal L(Y)=\nu} \, \mathbb E(|X-Y|) \, .$$ Actually we have
\begin{proposition}\label{propW}
Assume that $\mu$ and $\nu$ satisfy weak $(1,+\infty)$ Poincar\'e inequalities with respective rates $\beta_\mu$ and $\beta_\nu$. Then for all $s>0$, $$\beta_\nu(s) \, \leq \, \beta_\mu(s) \, + \, 2 \, W_1(\nu,\mu) \, .$$
\end{proposition}
\begin{proof}
Let $f$ be $1$-Lipschitz. We have 
\begin{eqnarray*}
\nu(|f-\nu(f)|) &\leq& \nu(|f-\mu(f)|) \, + \, |\nu(f)-\mu(f)| \\ &\leq& \, \mu(|f-\mu(f)|) \, + \, W_1(\nu,\mu) \, + \, |\nu(f)-\mu(f)|\\ &\leq& \beta_\mu(s) \, + \, 2 \, W_1(\nu,\mu) \, + s \, \Osc(f) \, .
\end{eqnarray*}
Here we used that $|\mu(|f-\mu(f)|)-\nu(|f-\mu(f)|)| \leq W_1(\mu,\nu)$ since $|f-\mu(f)|$ is still $1$-Lipschitz and that $|\nu(f)-\mu(f)|\leq W_1(\nu,\mu)$ too, i.e. the $W_1$ control for two different functions.
\end{proof}
We immediately deduce
\begin{corollary}\label{corWlog}
Let $\nu$ be a log-concave probability measure. Then for all $\mu$, $$C'_C(\nu) \, \leq \, \frac{16}{\pi} \, (C_C(\mu) +2 \, W_1(\nu,\mu)) \, .$$
\end{corollary}
\medskip

The proof of proposition \ref{propW} can be modified in order to give another approach of Theorem  \ref{corequiv}. Consider a Lipschitz function $f$ satisfying $0\leq f \leq 1=Osc (f)=\parallel f\parallel_\infty$ (recall remark \ref{remfortet}), then 
$$
\nu(|f-\nu(f)|) \leq \mu(|f-\mu(f)|) \, +  \, |\mu(|f-\mu(f)|)-\nu(|f-\mu(f)|)| \, + \, |\nu(f)-\mu(f)| \, .$$ Since $$\mu(|f-\mu(f)|) \leq \beta_\mu(s) \, \parallel |\nabla f|\parallel_\infty \, + \, s \, Osc(f) \, ,$$ while for $0\leq g \leq 1$, $$|\nu(g)-\mu(g)|=|\nu(g-\inf (g))-\mu(g-\inf(g))|\leq d_{TV}(\mu,\nu) \, \parallel g-\inf g\parallel_\infty = d_{TV}(\mu,\nu) \, \Osc(g) \, ,$$ we get applying the previous bound with $g=f$ and $g=|f-\mu(f)|$,
$$\nu(|f-\nu(f)|) \leq \beta_\mu(s) \, \parallel |\nabla f|\parallel_\infty + (s + 2 \, d_{TV}(\mu,\nu)) \, Osc(f) \, .$$
Hence, for any $\mu$ and $\nu$, for all $s'$ such that $s'> 2 \, d_{TV}(\mu,\nu)$, we have 
\begin{equation}\label{eqvariatany}
\beta_\nu(s') \leq \beta_\mu \left(s'-2 \, d_{TV}(\nu,\mu)\right) \, .
\end{equation}
We thus have shown
\begin{proposition}\label{propW0}
Let $\nu$ be a log-concave probability measure. Then for all $\mu$ such that $d_{TV}(\nu,\mu) \leq \, 1/4$,  we have for all $s < \frac 12 -  \, 2 d_{TV}(\mu,\nu)$, $$C'_C(\nu) \, \leq \, \frac{16 \, \beta_\mu(s)}{\pi \, (1-2s-4 \, d_{TV}(\mu,\nu))^2} \, .$$ In particular for $s=0$ we get 
$$C'_C(\nu) \, \leq \frac{16}{\pi \, (1-4 \, d_{TV}(\nu,\mu) )^2} \,  \, C_C(\mu)  \, .$$
\end{proposition}

Of course the disappointing part of the previous result is that, even if the distance between $\nu$ and $\mu$ goes to $0$, we cannot improve on the pre factor. In comparison with Theorem \ref{corequiv} we do not require $\mu$ to be log-concave, but the previous proposition is restricted to the case of not too big distance between $\mu$ and $\nu$ while we may take $\varepsilon$ close to $0$ in Theorem \ref{corequiv}. Notice that we really need the weak form of the $(1,\infty)$ inequality here (the non weak form of E. Milman is not sufficient), since we only get $\beta_\nu(s')$ for $s'$ large enough.
\medskip

\begin{remark}\label{remdesap}
The previous result with $\mu=\nu_A$, see proposition \ref{proprestrict}, furnishes a worse bound than in this proposition. \\ The fact that we have to use the total variation bounds for two different functions, prevents us to localize the method, i.e. to build an appropriate $\mu$ for each $f$ as in Eldan's localization method. Let us explain the previous sentence. \\ Pick some function $\beta$. Let $f$ be a given function satisfying $0\leq f \leq 1$. Assume that one can find a measure $\mu_f$ such that $\beta_{\mu_f} \leq \beta$ and $$|\mu_f(|f-\mu_f(f)|)-\nu(|f-\mu_f(f)|)| \, + \, |\nu(f)-\mu_f(f)| \leq \varepsilon \leq \frac 12 \, .$$ Then  $$\nu(|f-\nu(f)|) \leq \beta(s) \, \parallel |\nabla f|\parallel_\infty + (s + \varepsilon) \, Osc(f) \, $$ so that one can conclude as in proposition \ref{propW0}. Eldan's localization method is close to this approach, at least by controlling $|\nu(f)-\mu_f(f)|$, but not $|\mu_f(|f-\mu_f(f)|)-\nu(|f-\mu_f(f)|)|$. We shall come back to this approach later.
\hfill $\diamondsuit$
\end{remark}

\begin{remark}\label{rembl1}
In the proof of Proposition \ref{propW0} we may replace the total variation distance by the Bounded Lipschitz distance  $$d_{BL}(\mu,\nu)=\sup \{\mu(f)-\nu(f) \; \textrm{ for $f$ $1$-Lipschitz and bounded by $1$\}} \, ,$$ or the Dudley distance (also called the Fortet-Mourier distance) $$d_{Dud}(\mu,\nu)=\sup \{\mu(f)-\nu(f) \; \textrm{ for $\parallel f\parallel_\infty + \parallel|\nabla f|\parallel_\infty \leq  1$\}} \, .$$ Recall that $$d_{Dud}(\nu,\mu) \leq d_{BL}(\nu,\mu) \leq 2 \, d_{Dud}(\nu, \mu) \, .$$ Provided one replaces $\parallel g\parallel_\infty$ by $\parallel g\parallel_\infty + \parallel |\nabla g|\parallel_\infty$, \eqref{eqvariatany} is replaced by 
\begin{equation}\label{eqblr1}
\beta_\nu(s') \leq \beta_\mu(s'-2d_{Dud}(\nu,\mu)) + 2 d_{Dud}(\nu,\mu) \, ,
\end{equation}
so that, when $\nu$ is log-concave, we get 
\begin{equation}\label{eqblr2}
C'_C(\nu) \leq  \, \frac{16 \, (\beta_\mu(s) + 2 d_{Dud}(\nu,\mu))}{\pi \, (1-2s-4 \, d_{Dud}(\mu,\nu))^2} \, .
\end{equation}
One can of course replace $d_{Dud}$ by the larger $d_{BL}$ in these inequalities.\\
When $\mu$ and $\nu$ are isotropic log-concave probability measures, it is known that $$d_{TV}(\mu,\nu) \leq C \, \sqrt{n \, d_{BL}(\mu,\nu)}$$ according to proposition 1 in \cite{Meckes2}. Combined with corollary \ref{corequiv} this bound is far to furnish the previous result since it gives a dimension dependent result. In addition we do not assume that $\mu$ is log concave. \hfill $\diamondsuit$
\end{remark}
\begin{remark}\label{remdist}
Remark that with our definitions $d_{Dud}\leq d_{BL} \leq 2 \, d_{TV} \leq 2$. \hfill $\diamondsuit$
\end{remark}
\medskip

\section{Mollifying the measure.}\label{molly}

In this section we shall study inequalities for mollified measures. If $Z$ is a random variable we will call mollified variable the sum $Z+X$ where $X$ is some independent random variable, i.e. the law $\nu_Z$ is replaced by the convolution product $\nu_Z*\mu_X$. In this situation it is very well known that $$C_P(\sqrt{\lambda} \, Z \, + \, \sqrt{1-\lambda} \, X) \, \leq \, \lambda \, C_P(Z) + (1-\lambda) \, C_P(X)$$ for $0 \leq \lambda \leq 1$ (see \cite{BBN}). Taking $\lambda=1/2$ it follows that 
\begin{equation}\label{eqconvol}
C_P(Z+X) \leq C_P(Z)+C_P(X) \, .
\end{equation}
It is well known that mollifying the measure can improve on functional inequalities. For instance if $\nu$ is a compactly supported probability measure, the convolution product of $\nu$ with a gaussian measure will satisfy a logarithmic Sobolev inequality as soon as the variance of the gaussian is large enough (see e.g. \cite{Zim}), even if $\nu$ does not satisfy any ``interesting'' functional inequality (for instance if $\nu$ has disconnected support); but the constant is desperately dimension dependent. We shall see that adding the log-concavity assumption for $\nu$ will help to improve on similar results.
\medskip

\subsection{Mollifying using transportation.\\ \\}\label{subseckls}  A first attempt to look at mollified measures was done by Klartag who obtained the following transportation inequality on the hypercube in \cite{Klarcube}:
\begin{theorem}\label{thmklarcube}[Klartag]
Let $R\geq 1$ and let $Q$ be some cube in $\mathbb R^n$ of side length $1$ parallel to the axes. Let $\mu=p(x)dx$ be a log-concave probability measure on $Q$ satisfying in addition 
\begin{equation}\label{eqconvex}
p(\lambda x+(1-\lambda)y) \leq R \, (\lambda p(x) + (1-\lambda) p(y))
\end{equation}
 for any $0\leq \lambda \leq 1$ and any pair $(x,y) \in Q^2$ such that all cartesian coordinates of $x-y$ are vanishing except one ($x-y$ is proportional to some $e_j$ where $e_j$ is the canonical orthonormal basis).\\ Then, $\mu$ satisfies a $T_2$ Talagrand inequality, i.e. there exists some $C$ (satisfying $C \leq 40/9$) such that for any $\mu'$, $$W_2^2(\mu',\mu) \leq C \, R^2 \, D(\mu'||\mu) \, ,$$ where $W_2$ denotes the Wasserstein distance and $D(.||.)$ the relative entropy.\\ In particular $\mu$ satisfies a Poincar\'e inequality with $C_P(\mu) \leq \frac{C \, R^2}{2}$.
\end{theorem}
The final statement is an easy and well known consequence of the $T_2$ transportation inequality, as remarked in \cite{Klarcube} corollary 4.6.
\medskip

In the sequel $Q$ will denote the usual unit cube $[-\frac 12 \, , \, \frac 12]^n$, and for $\theta>0$, $\theta Q$ will denote its homothetic image of ratio $\theta$. 

For $\theta>0$, let $Z_\theta$ be a log-concave random vector whose distribution $\mu_\theta=p_\theta(x) \, dx$ is supported by $\theta Q$ and satisfies \eqref{eqconvex} for any pair $(x,y) \in (\theta Q)^2$ and some given $R$. Then the distribution $\mu$ of $Z=Z_\theta/\theta$ satisfies the assumptions in Theorem \ref{thmklarcube} with the same $R$ since its probability density is given for $x \in Q$ by $$p(x) = \theta^{n} \, p_\theta(\theta \, x) \, .$$ In particular $C_P(Z) \leq \frac 12 \, CR^2$ so that $C_P(Z_\theta) \leq \frac 12 \, CR^2 \, \theta^2$. \\ In the sequel we can thus replace $Q$ by $\theta Q$. We shall mainly look at two examples of such $p$'s: convolution products with the uniform density and the gaussian density.
\smallskip

\subsubsection{Convolution with the uniform distribution.}\label{subsubunif}

Consider $U_\theta$ a uniform random variable on $\theta Q$, $\theta>1$. Its density $p(x)=\theta^{-n} \, \mathbf 1_{\theta Q}(x)$ satisfies \eqref{eqconvex} in $\theta Q$ with $R=1$. It is immediate that $$p(\lambda x + (1-\lambda)x' - y)\leq \, \lambda \, p( x  - y) + (1-\lambda) \, p(x'-y)$$ for all $x,x',y$ such that $x-y$ and $x'-y$ belong to $\theta Q$. \\ Let $Z$ be a log-concave random variable whose law $\mu$ is supported by $Q$. The law $$\nu_\theta(dx)  =\left(\int \, p(x-y) \, \nu(dy)\right) \, dx$$ of $Z+U_\theta$ is still log-concave according to Prekopa-Leindler and satisfies \eqref{eqconvex} with $R=1$ on $(\theta -1) Q$. According to what precedes, its restriction $$\nu_{\theta,1}(dx) = \frac{\mathbf 1_{(\theta-1)Q}}{\nu_\theta((\theta-1)Q)} \, \nu_\theta(dx)$$ to $(\theta-1) Q$ satisfies $$C_P(\nu_{\theta,1}) \leq \frac 12 \, C \, (\theta-1)^2 \, .$$ Thus $$C'_C(\nu_{\theta,1}) \leq \frac{6}{\sqrt 2}\, \sqrt C \, (\theta-1) \, ,$$ according to Ledoux's comparison result.  

\medskip

\subsubsection{Convolution with the gaussian distribution.}\label{subsubgauss}

Let $\gamma_\beta(x)= (2\pi \beta^2)^{-n/2} \, \exp \left(- \, \frac{|x|^2}{2\beta^2}\right)$ be the density of a centered gaussian variable $\beta G$ (as before $G$ is the standard gaussian). \\ It is elementary to show that $\gamma_\beta$ satisfies the following convexity type property close to \eqref{eqconvex}: for all pair $(x,x')$ and all $\lambda \in [0,1]$,
\begin{equation}\label{eqgauss}
\gamma_\beta(\lambda x + (1-\lambda)x') \leq e^{\frac{|x-x'|^2}{8\beta^2}} \, \left(\lambda \, \gamma_\beta(x) + (1-\lambda) \, \gamma_\beta(x')\right) \,  \, ,
\end{equation}
this inequality being optimal and attained for pairs $(x,x')=(x,-x)$.\\ It immediately follows that $$\gamma_\beta(\lambda x + (1-\lambda)x'-y) \leq e^{\frac{|x-x'|^2}{8\beta^2}} \, \left(\lambda \, \gamma_\beta(x-y) + (1-\lambda) \, \gamma_\beta(x'-y)\right) \,  \, ,$$ for all $(x,x',y)$. It follows that for all log-concave random vector $Z$, the distribution of $Z+\beta G$ is still satisfying \eqref{eqgauss}. We have thus obtained
\begin{proposition}\label{lemklarconv}
Let $\beta$ and $\theta$ be positive real numbers. Let $Z$ be some log-concave random vector and denote by $\nu_\beta$ the distribution of $Z+\beta G$. Then the restriction $$\nu_{\beta,\theta}= \frac{\BBone_{\theta Q}}{\nu_\beta(\theta Q)} \, \, \nu_\beta$$ satisfies $$C_P(\nu_{\beta,\theta}) \leq \frac{20}{9} \, \, \theta^2 \, e^{\frac{\theta^2}{8 \beta^2}} \, .$$ 
\end{proposition}
\smallskip

Notice that if we let $\beta$ go to infinity, $\nu_{\beta,\theta}$ converges to the uniform measure on $\theta Q$ so that the order $\theta^2$ is the good one (if not the constant $\frac{20}{9}$). Also notice that if $\nu$ is supported by $\alpha Q$, we may replace $\beta G$ by a random variable whose distribution is $\gamma_\beta$ restricted to $(\alpha+\theta)Q$.
\medskip

The control obtained in proposition \ref{lemklarconv} (or in the uniform case) can be very interesting for practical uses, in particular for applied statistical purposes using censored random variables. But if we want to use it in order to get some information on the original $\nu$, the result will depend dramatically on the dimension, since $\nu_{\beta}(\theta Q)$ is very small.
\bigskip

\subsection{Mollifying with gaussian convolution. A stochastic approach.\\ \\}\label{subsecOU}

In this section we shall introduce our approach for controlling the Poincar\'e constant using some appropriate stochastic process. To this end, we first consider a standard Ornstein-Uhlenbeck process $X_.$, i.e. the solution of
\begin{equation}\label{eqOU}
dX_t = dB_t - \frac 12 \, X_t \, dt \, .
\end{equation}
The law of $X_t^x$ (i.e. the process starting from point $x$) will be denoted by $G(t,x,.)=\gamma(t,x,.) \, dx$, $\gamma(t,x,.)$ being thus the density of a gaussian random variable with mean $e^{-t/2} \, x$ and covariance matrix $(1-e^{-t}) \, Id$. The standard gaussian measure $\gamma$ is thus the unique invariant (reversible) probability measure for the process. The law of the O-U process starting from $\mu$ will be denoted by $\mathbb G_\mu$.
\medskip

Let $\nu(dx)=e^{-V(x)} \, dx$ be a probability measure, $V$ being smooth. We assume that $\nu$ is log-concave. Let $$h(x)=(d\nu/d\gamma)(x)=(2\pi)^{n/2} \, e^{((|x|^2/2)-V(x))} \, .$$ 
The relative entropy $D(\nu||\gamma)$ satisfies $$D(\nu||\gamma) = \frac n2 \, \log(2\pi) + \int ((|x|^2/2)-V(x)) \, \nu(dx) < +\infty \, ,$$ since $V$ is non-negative outside some compact subset.
We may define for all $t$ (Mehler formula), $$\E(h(X_t^x)) = G_t h(x) = \int \, h(e^{-t/2} x + \sqrt{1-e^{-t}}y) \, \gamma(dy) \, ,$$ which is well defined, positive, smooth ($C^{\infty}$), and satisfies for all $t>0$, $$\partial_t G_th(x)=\frac 12 \Delta G_th(x) - \frac 12 \, \langle x,\nabla G_th(x)\rangle \, .$$
In particular we may consider the solution of
\begin{equation}\label{eqOUpert}
dY_t = dB_t - \frac 12 \, Y_t \, dt \, + \nabla \log G_{T-t}h(Y_t) \, = \, dB_t + b(t,Y_t) \, dt \, ,
\end{equation}
for $0\leq t \leq T$, with initial distribution $G_T h \, \gamma$. This stochastic differential equation  has a strongly unique solution (since the coefficients are smooth) up to some explosion time $\xi$. Since $D(\nu||\gamma)<+\infty$, this explosion time is almost surely infinite, or if one prefers the solution of \eqref{eqOUpert} is defined up to and including time $T$. In addition the law $\Q$ of the solution satisfies $$\frac{d\Q}{d\G_\gamma}|_{\mathcal F_T} = h(\omega_T) \, .$$ In particular, $$\nu_s = \Q \circ \omega_s^{-1} = G_{T-s}h \, \gamma \quad \textrm{ for all $0\leq s \leq T$,} $$ thanks to the stationarity of $\mathbb G_\gamma$. Of course this is nothing else but the $h$-process corresponding to $h$ and the O-U process, but with a non bounded $h$. If one prefers, $\Q$ is simply the law of the time reversal, at time $T$, of an Ornstein Uhlenbeck process with initial law $\nu$. For more details see e.g. \cite{CGsp}.
\medskip

A specific feature of the O-U process is that, according to Prekopa-Leindler theorem, $\nu_s$ is still a log-concave measure as the law of $e^{-(T-s)/2}Z + \sqrt{1-e^{-(T-s)}} G$ where $Z$ and $G$ are two independent random variables with respective distribution $\nu$ and $\gamma$ (so that the pair $(Z,G)$ has a log-concave distribution). This property of conservation of log-concave distributions is typical to the ``linear'' diffusion processes of O-U type (see \cite{Koles}). Hence $\nu_s(dx) = e^{-V_s(x)} \, dx$ for some potential $V_s$ which is smooth thanks to the Mehler formula we recalled before, and convex.\\ 
It follows that $G_{T-s}h \, \gamma$ is log-concave, and finally  that $b(t,.)$ satisfies the curvature condition
\begin{equation}\label{curvb}
2 \, \langle \, b(t,x)-b(t,y) \, , \, x-y \, \rangle \, \leq \, |x-y|^2 \, ,
\end{equation}
for all $t$, $x$ and $y$. \eqref{curvb} is called condition (H.C.-1) in \cite{CGsp}. It follows that for all $T>0$,
\begin{equation}\label{eqlogconc1}
C_P(\nu) \, \leq \, e^T \, C_P(G_Th \, \gamma) + (e^T -1) \, .
\end{equation}
For time homogeneous drifts this is nothing else but the so called ``local'' Poincar\'e inequality in \cite{BaGLbook} (Theorem 4.7.2). The extension to time dependent drift is done in \cite{CGsp} Theorem 5. We have thus obtained

\begin{theorem}\label{propsmooth}
Let $Z$ be a random variable with log-concave distribution $\nu$ and $G$ be a standard gaussian random variable independent of $Z$. Then for all $0<\lambda\leq 1$ it holds $$C_P(Z) \, \leq \, \frac{1}{\lambda} \, C_P(\sqrt{\lambda} \, Z + \sqrt{1-\lambda} \, G) + \left(\frac{1}{\lambda} -1\right) \, .$$
\end{theorem}
Of course since $C_P(aY)=a^2 \, C_P(Y)$ for any random variable and any $a\in \R$, we obtain that for all real numbers $\alpha$ and $\beta$, defining $\lambda=\frac{\alpha^2}{\alpha^2+\beta^2}$,
\begin{equation}\label{eqlogconc2}
C_P(Z) \leq \frac{1}{\alpha^2} \, \, C_P(\alpha \, Z \, + \, \beta \, G) + \frac{\beta^2}{\alpha^2} \, .
\end{equation}
\smallskip
Using all the comparison results we already quoted, it follows
\begin{equation}\label{eqlogconc3}
C'_C(Z) \leq \frac{12}{\alpha} \, \, C'_C(\alpha Z + \beta G) \, + \, \frac{6\beta}{\alpha} \, .
\end{equation}
\smallskip

\begin{remark}\label{remarkOUdiag}
In the proof we may replace the standard Ornstein-Uhlenbeck process by any O-U process $$dX_t = dB_t - \frac 12 \, A \, X_t \, dt \, ,$$ where $A$ is some non-negative symmetric matrix. Up to an orthogonal transformation we may assume that $A$ is diagonal with non-negative diagonal terms $a_i$. Assume that $$\max_{i=1,...,n} a_i
 \leq 1 \, .$$ Then \eqref{curvb} is still satisfied, and no better inequality is. Similarly the Poincar\'e constant of the corresponding gaussian variable $G'$ is unchanged so that \eqref{eqlogconc1} is still satisfied, and no better inequality is. We may thus in \eqref{eqlogconc2} replace the standard gaussian vector $G$ by any centered gaussian vector, still called $G$, with independent entries $G_i$ such that $\Var(G_i)\leq 1$ for all $i$; in particular a degenerate gaussian vector. \hfill $\diamondsuit$
\end{remark}
\smallskip

\begin{remark}\label{remarkbrownOU}
\eqref{eqlogconc2} with $\alpha=1$ and $\beta=\sqrt T$, can also be derived using a similar time reversal argument for the Brownian motion (with reversible measure Lebesgue) instead of the O-U process. The corresponding drift $b$ satisfies $\langle \, b(t,x)-b(t,y) \, , \, x-y \, \rangle \, \leq \, 0$ which directly yields the result. Unfortunately the invariant measure is no more a probability measure. \hfill $\diamondsuit$
\end{remark}
\smallskip

\begin{remark}\label{remouvar}
Theorem \ref{propsmooth} can be viewed as a complement of previous similar results comparing the behavior of log-concave distributions and their gaussian mollification. Indeed, if $Z$ is an isotropic log-concave probability measure and $G$ a standard gaussian vector, it is elementary to show that for all $t \geq 0$,
\begin{equation}\label{eqvarconvgauss}
\Var(|Z+\sqrt t \, G|^2) = \Var(|Z|^2) \, + \, 2nt(2+t) \, ,
\end{equation}
so that if for some $t_0$, $\Var(|Z+\sqrt {t_0} \, G|^2) \leq C n$ then the same bound is true for $\Var(|Z|^2)$ i.e. the variance conjecture is true. Similarly, it is recalled in \cite{GEM} (just before formula (4.5) therein), that a Fourier argument due to Klartag furnishes a control for the deviations $$\mathbb P(|Z|>(1+t)\sqrt n) \leq C \, \mathbb P(|Z+G|>(1+(1+t)^2)\sqrt n)$$ and $$\mathbb P(|Z|<(1-t)\sqrt n) \leq C \, \mathbb P(|Z+G|<(1+(1-t)^2)\sqrt n)$$ for some universal constant $C$. \hfill $\diamondsuit$
\end{remark}
\medskip

\section{Probability metrics and log-concavity.}\label{secmetrics}

\subsection{Comparing Bounded Lipschitz and Total Variation distances.  \\ \\}\label{subsecBLvsTV}
Let $\mu$ and $\nu$ be two probability measures. It is immediate to show the analogue of \eqref{eqconvol}, i.e. if $\mu*\nu$ denotes the convolution product of both measures
\begin{equation}\label{eqconvbeta}
\beta_{\mu*\nu}(s) \leq \beta_\mu(s/2) \, + \, \beta_\nu(s/2) \, ,
\end{equation}
here $\beta$ corresponds to the usual centering with the mean (not the median).\\ Let $0<t$. Denote by $\gamma_t$ the distribution of $tG$, that is the gaussian distribution with covariance $t^2 Id$ (whose density is $\tilde \gamma_t$). For $0\leq g \leq 1$, one has $$(\nu*\gamma_t) (g) = \nu(g*\tilde \gamma_t)$$ and $g_t=g*\tilde \gamma_t$ is still bounded by $1$ and $1/t$-Lipschitz (actually $\sqrt{2}/t\sqrt{\pi}$) according to \eqref{eqledgap2} applied to the Brownian motion semi-group at time $t^2$. It follows that $$(\nu*\gamma_t) (g) - (\mu*\gamma_t) (g) = \nu(g_t)-\mu(g_t)$$ so that
\begin{eqnarray}\label{eqdistanceconvol}
d_{TV}(\mu*\gamma_t,\nu*\gamma_t) & \leq & \left(1 \, \vee \, \frac{\sqrt{2}}{t\sqrt{\pi}}\right) \, d_{BL}(\mu,\nu) \nonumber \\ & & \qquad \textrm{ and } \\  d_{TV}(\mu*\gamma_t,\nu*\gamma_t) & \leq & \left(1 \, + \, \frac{\sqrt{2}}{t\sqrt{\pi}}\right) \, d_{Dud}(\mu,\nu) \, . \nonumber
\end{eqnarray}
We may thus apply \eqref{eqvariatany}, and the fact that $C'_C(\gamma_t)= t$ in order to get, provided respectively $$s>2(1\vee \sqrt 2/t \sqrt \pi) \, d_{BL}(\nu,\mu) \textrm{ or } s>2(1 + \sqrt 2/t \sqrt \pi) \, d_{Dud}(\mu,\nu) \, ,$$ the following
\begin{eqnarray}\label{eqvariatanybis}
\beta_{\nu*\gamma_t}(s) &\leq& \beta_{\mu*\gamma_t}\left(s- \, 2(1\vee \sqrt{2}/t \sqrt{\pi}) \,  d_{BL}(\mu,\nu)\right) \nonumber \\ &\leq& \beta_\mu\left(\frac s2- \, (1\vee \sqrt{2}/t \sqrt{\pi}) \,  d_{BL}(\mu,\nu)\right) \, + \, t \, ,
\end{eqnarray}
or $$\beta_{\nu*\gamma_t}(s) \leq \, \beta_\mu\left(\frac{s}{2} - \, (1 + \sqrt 2/t \sqrt \pi) \, d_{Dud}(\mu,\nu)\right) \, + \,  t \, .$$
Gathering the previous inequality, Theorem \ref{thmweakmil} and \eqref{eqlogconc2}, we get new  versions of proposition \ref{propW0}, slightly different from the one we gave in Remark \ref{rembl1}. \\ For example, for a log-concave $\nu$ and for all $\mu$ such that $d_{BL}(\nu,\mu) \leq \, 1/4$, if we choose $t=\sqrt 2/\sqrt \pi$ it holds $$C'_C(\nu*\gamma_t) \, \leq \, \frac{16 \, (C'_C(\mu)+\sqrt{2/\pi})}{\pi(1-4 \, d_{BL}(\mu,\nu))^2} \, ,$$ so that $$C_P(\nu) \leq \frac 2\pi \, + \, 4 \, \left(\frac{16 \, (C'_C(\mu)+\sqrt{2/\pi})}{\pi(1-4 \, d_{BL}(\mu,\nu))^2}\right)^2 \, .$$
\medskip

But we may use \eqref{eqdistanceconvol} in a potentially more interesting direction. Indeed, using theorem \ref{corequiv} and provided $\nu$ and $\mu$ (hence $\nu*\gamma$ and $\mu*\gamma$) are log-concave
\begin{equation}\label{eqdtvgamma}
C'_C(\nu*\gamma) \, \leq \, \frac{\kappa}{(1-d_{BL}(\nu,\mu))^2} \, \left(1 \vee \, \ln(1/(1-d_{BL}(\nu,\mu)))\right) \, C'_C(\mu*\gamma) \, ,
\end{equation}
for some $\kappa \leq 192 e/\pi$, provided $d_{BL}(\nu,\mu) \leq 1$ (we have skipped the $\sqrt{2/\pi}$ for simplicity). Hence using \eqref{eqlogconc3}, the comparison between $C_P$ and $C'_C$ and \eqref{eqconvol} we have obtained a partial analogue of theorem \ref{corequiv} with the weaker  bounded Lipschitz distance
\begin{theorem}\label{thmklsbl}
Let $\nu$ and $\mu$ be two log-concave probability measures on $\mathbb R^n$, $\gamma$ be the standard gaussian distribution on $\mathbb R^n$.
\begin{enumerate}
\item \quad Then if $d_{BL}(\nu,\gamma)=1-\varepsilon$, 
$$C'_C(\nu) \, \leq  \, \frac{C}{\varepsilon^2} \, \left(1 \vee \, \ln(1/\varepsilon) \right)\, + 6 \, ,$$ where the universal constant $C$ can be chosen less than $\frac{13824 \, e}{\pi}$.
\item If $d_{BL}(\nu,\mu)=1-\varepsilon$ then $$C'_C(\nu) \, \leq  \, \frac{D}{\varepsilon^2} \, \left(1 \vee \, \ln(1/\varepsilon)\right) \, (2 C'_C(\mu) +1) \, + 6 \, ,$$ where the universal constant $D$ can be chosen less than $6C$.
\end{enumerate}
\end{theorem}
\medskip

\subsection{Comparing Total Variation and $W_1$.  \\ \\}\label{subsecTVvsW1}
We still use the notation $G_T$ for the Ornstein-Uhlenbeck semi-group we introduced in \eqref{eqOU}, in particular we write $G_T\nu$ for the law at time $T$ of the O-U process with initial distribution $\nu$. \\
It is well known that $W_1(G_T\nu,G_T\mu) \leq e^{-T/2} \, W_1(\nu,\mu)$.  Recall that this contraction property is an immediate consequence of synchronous coupling, i.e. if we build two solutions $X_.$ and $Y_.$ of \eqref{eqOU} with the same Brownian motion it holds $$|X_t-Y_t| = |X_0-Y_0| - \frac 12 \, \int_0^t |X_s-Y_s| ds$$ implying the result by choosing an optimal coupling $(X_0,Y_0)$. 
\medskip

If we want to replace the $W_1$ distance by the total variation distance (or the bounded Lipschitz distance) one has to replace the synchronous coupling by a coupling by reflection following the idea of Eberle (\cite{Ebe2}) we already used in \cite{CGsp}. This yields (see \cite{CGsp} subsection 7.4) the following inequality 
\begin{equation}\label{eqconvVT}
d_{TV}(G_T \nu,G_T\mu) \leq \frac{e^{-T/2}}{\sqrt {2\pi \, (1-e^{-T})}} \; W_1(\nu,\mu) \, .
\end{equation}

What we did before allows us to state a negative result
\begin{proposition}\label{OUcontractepas}
The Ornstein-Uhlenbeck semi-group is not a contraction in total variation distance, nor in bounded Lipschitz distance.
\end{proposition}
\begin{proof}
Since $d_{TV}\leq 1$ and applying the semi-group property, any uniform decay $d_{TV}(G_T \nu,\gamma)  \leq h(T)$ with $h$ going to $0$ implies an exponential decay and the contraction property, for some $T>0$, $d_{TV}(G_{T} \nu,\gamma)  \leq \frac 12$. If $\nu_\lambda$ is the log concave distribution of some random vector $\lambda \, X$, it follows that for universal constants $C$ and $C'$, $$C_P(\nu_\lambda) \leq C \, (C_P(G_T \nu_\lambda)+1) \leq C' \, .$$ But $C_P(\nu_\lambda)=\lambda^2 \, C_P(\nu_1)$ yielding a contradiction.\\ Similarly, if $d_{BL}(G_T\nu,\gamma)\leq 1/2$ then $d_{TV}(G_{T+1}\nu,\gamma)\leq \frac 12$ which is impossible.
\end{proof}
\medskip

\subsection{Comparison with other metrics on probability measures. \\ \\} \label{subsecmetrics}
The weakest distance we introduced is the Bounded Lipschitz distance. It is known that this distance metrizes weak convergence. We may thus compare $d_{BL}$ with the L\'evy-Prokhorov distance. 
\begin{definition}\label{defLP}
If $\mu$ and $\nu$ are two probability measures the L\'evy-Prokhorov distance $d_{LP}(\mu,\nu)$ is defined as $$d_{LP}(\mu,\nu) = \inf \{\varepsilon \geq 0 \; ; \; \mu(A) \leq \nu(A+B(0,\varepsilon)) + \varepsilon \, ; \textrm{ for all Borel set $A$.} \}$$ $A+B$ is as usual the set of $a+b$ where $a\in A$ and $b \in B$, and $B(x,u)$ denotes the euclidean ball of center $x$ and radius $u$. 
\end{definition}
It is well known that $d_{LP}$ is a metric (in particular $d_{LP}(\mu,\nu)=d_{LP}(\nu,\mu)$ despite the apparent non symmetric definition, just tacking complements), clearly bounded by $1$, that actually metrizes the convergence in distribution (weak convergence). We may also replace Borel sets $A$ by closed sets $A$, hence defining $\rho(\mu,\nu)$ and symmetrizing the definition i.e. $d_{LP}(\mu,\nu)=max(\rho(\mu,\nu), \rho(\nu,\mu))$. \\
The following properties can be found in \cite{Dud} Corollaries 2 and 3 (and using that $d_{LP}$ is less than 1), or \cite{Dudbook} problem 5 p.398 or in \cite{rachev} (with a worse constant)
\begin{proposition}\label{propprokh}
\begin{enumerate}
\item It holds $$\frac 14 \, d_{BL}(\nu,\mu) \leq \frac 12 \, d_{Dud}(\nu,\mu) \leq d_{LP}(\nu,\mu) \leq  \, \sqrt{\frac 32 \, d_{Dud}(\nu,\mu)} \leq  \sqrt{\frac 32 \, d_{BL}(\nu,\mu)} \, ,$$
\item $d_{LP}(\nu,\mu) = \inf \{K(X,Y) \; ; \; \mathcal L(X)=\nu \; , \; \mathcal L(Y)=\mu\}$ where $$K(X,Y)= \inf \{\varepsilon \geq 0 \; ; \; \mathbf P(|X-Y|>\varepsilon)\leq \varepsilon\}$$ is the Ky-Fan distance between $X$ and $Y$, and $\mathcal L(X)$ denotes the probability distribution of $X$.
\end{enumerate}
\end{proposition}
Assume that $\mu$ and $\nu$ are log-concave. Then $\mu\otimes \nu$ is also log-concave according to Prekopa-Leindler theorem, so that $(x,y) \mapsto x-y$ is a polynomial of degree 1 on $\mathbb R^n \otimes \mathbb R^n$. For such a polynomial, moment controls have been obtained by several authors. We shall use the version in \cite{Frad} Corollary 4 (see references therein too)
\begin{theorem}\label{thmfrad}[Fradelizi] \; 
If $\eta$ is a log-concave probability measure and $P$ is a polynomial of degree $1$, then for all $c>0$ and all $t\geq 1$, $$\eta(x \, ; \, |P(x)|>ct) \leq \eta(x \, ; \, |P(x)|>c)^{\frac{1+t}{2}} \, ,$$ provided the left hand side is (strictly) positive.
\end{theorem}
If $X$ and $Y$ are independent log-concave random variables, we deduce that for $t\geq 1$,
\begin{equation}\label{eqdevfrad}
\mathbb P (|X-Y|>t \, K(X,Y)) \, \leq \, (K(X,Y))^{\frac{1+t}{2}} \, .
\end{equation}
Using $$\mathbb E(|X-Y|) = \int_0^{+\infty} \, \mathbb P (|X-Y|>t) \, dt$$ we have thus obtained
\begin{equation}\label{eqwfrad}
\mathbb E(|X-Y|) \, \leq \, K(X,Y) \, \left(1 \, + \, \frac{2 \, K(X,Y)}{\ln (1/K(X,Y))}\right) \, ,
\end{equation}
so that taking an optimal coupling on the right hand side we have obtained
\begin{corollary}\label{corfrad}
Let $\mu$ and $\nu$ be two log-concave probability measures. Then $$d^2_{LP}(\mu,\nu) \, \leq W_1(\mu,\nu) \, \leq \, d_{LP}(\mu,\nu) \, \left(1 \, + \, \frac{2 \, d_{LP}(\mu,\nu)}{\ln (1/d_{LP}(\mu,\nu))}\right) \, ,$$ so that, $$C'_C(\nu) \leq \frac{16}{\pi} \, \left(C_C(\mu)+ 2 \, d_{LP}(\mu,\nu) \, \left(1 \, + \, \frac{2 \, d_{LP}(\mu,\nu)}{\ln (1/d_{LP}(\mu,\nu))}\right)\right) \, .$$
\end{corollary}
Recall that the left hand side of the inequality between distances is always true (see e.g. \cite{bob16} (10.1) p.1045). \\ Combining all what precedes we have also obtained
\begin{corollary}\label{corconvprok}
The Ornstein-Uhlenbeck semi-group is not a contraction in L\'evy-Prokhorov distance. However if $\nu$ is log-concave $$d_{LP}(G_T\nu,\gamma) \, \leq \, e^{-T/4} \, \left[d_{LP}(\gamma,\nu) \, \left(1 \, + \, \frac{2 \, d_{LP}(\gamma,\nu)}{\ln (1/d_{LP}(\gamma,\nu))}\right)\right]^{\frac 12} \, .$$
\end{corollary}
\medskip

Provided $d_{BL}(\mu,\nu) \leq 2/3$ we also have,  $$W_1(\mu,\nu) \, \leq \,  \sqrt{\frac 32 \, d_{BL}(\mu,\nu)} \, \left(1 \, + \, \frac{\sqrt{6 d_{BL}(\mu,\nu)}}{\ln (\sqrt 2/\sqrt{3 d_{BL}(\mu,\nu)})}\right) \, ,$$ so that 
we get a new bound for two log-concave measures $$C'_C(\nu) \leq \frac{16}{\pi} \, \left(C_C(\mu)+  \, \sqrt{6 d_{BL}(\mu,\nu)} \, \left(1 \, + \, \frac{\sqrt{6 d_{BL}(\mu,\nu)}}{\ln (\sqrt 2/\sqrt{3 d_{BL}(\mu,\nu)})}\right)\right) \, .$$ For small values of $d_{BL}(\mu,\nu)$, this bound is better than \eqref{eqblr2} (but here we need both measures to be log-concave) and theorem \ref{thmklsbl}, which is true for large values of $d_{BL}(\mu,\nu)$. \\ One can also compare corollary \ref{corfrad} with proposition 4 in \cite{Meckes2} giving a dimensional inequality $$W_1(\mu,\nu) \, \leq \, C \, \left( \sqrt n \, \vee \, \ln\left(\frac{\sqrt n}{d_{BL}(\mu,\nu)}\right) \right) \, d_{BL}(\mu,\nu) \, ,$$ for isotropic log-concave probability measures. 
\medskip

\begin{remark}\label{remdesapoint}
The previous results give some hints on the (somehow bad) structure of isotropic log-concave measures. Indeed look, on one hand at the uniform measure $\mu^n$ on $A=[-\sqrt 3 \, , \, \sqrt 3]^{\otimes n}$ associated to a random variable $U$, on the other hand at the standard gaussian distribution $\gamma^n$ associated to $G$. $\mu^n(A)=1$ when $$\gamma^n(A+B(0,\varepsilon)) \leq \gamma^n([- \sqrt 3 -\varepsilon \, , \, \sqrt 3 + \varepsilon]^{\otimes n}) = (\gamma^1([-\sqrt 3 -\varepsilon \, , \, \sqrt 3 + \varepsilon]))^n \, ,$$ so that $$d_{LP}(\mu^n,\gamma^n) \, \geq \, 1 - u^n$$ with $u=\gamma^1([- \sqrt 3  \, , \, \sqrt 3])$, hence, for large $n$, $d_{LP}(\mu^n,\gamma^n)$ is as close to $1$ as we want. Consequently we cannot expect to get a dimension free nice upper bound for the L\'evy-Prokhorov distance. The question is then whether such a bound is true if we consider the set of $\nu*\gamma_t$ where $\nu$ describes the set of \emph{isotropic} log-concave distributions, or not. We know that such a bound does not exist for all log-concave distributions, according to Corollary \ref{corconvprok}. \hfill $\diamondsuit$
\end{remark}
\medskip

\medskip

If the L\'evy-Prokhorov distance seems difficult to estimate, one can relate it to a Wasserstein distance for a new distance. Introduce
\begin{equation}\label{LPtrans}
W_{LP}(\nu,\mu) = \inf\left\{\int \, \frac{|x-y|}{1+|x-y|} \, \pi(dx,dy) \; ; \; \pi\circ x^{-1}=\nu \; , \; \pi\circ y^{-1}=\mu  \right\} \, .
\end{equation}
\begin{proposition}\label{propKF}
Let $K^*(X,Y)=\mathbb E\left(\frac{|X-Y|}{1+|X-Y|}\right)$. It holds $$\frac 12 \, K^*(X,Y) \leq  \, K(X,Y) \, \leq \, \sqrt{ 2 K^*(X,Y)} \, .$$ Consequently $$\frac 12 \, W_{LP}(\mu,\nu) \, \leq \, d_{LP}(\mu,\nu) \, \leq \, \sqrt{ 2 \, W_{LP}(\mu,\nu)} \, .$$
\end{proposition}
\begin{proof}
Denote $Z=|X-Y|$ and $Z^*=\frac{Z}{1+Z}$, so that $Z=\frac{Z^*}{1-Z^*}$. \\ On one hand, since $Z^* \leq 1$,  $$\mathbb E(Z^*) \leq \mathbb P(Z^*>\eta) + \eta \, .$$ But $\mathbb P(Z^*>\eta)=\mathbb P(Z>\varepsilon)$ for $\eta=\frac{\varepsilon}{1+\varepsilon}$, so that $$\mathbb E(Z^*) \leq K(X,Y) + \frac{K(X,Y)}{1+K(X,Y)} \leq 2 \, K(X,Y) \, .$$ Conversely, using what precedes and Markov inequality, $$\mathbb P(Z>\varepsilon) \leq \frac{\mathbb E(Z^*)}{\frac{\varepsilon}{1+\varepsilon}}$$ so that $\mathbb P(Z>\varepsilon) \leq \varepsilon$ provided $\mathbb E(Z^*) \leq \frac{\varepsilon^2}{1+\varepsilon}$ in particular if $\mathbb E(Z^*) \leq \frac{\varepsilon^2}{2}$ (because we only have to consider $\varepsilon \leq 1$) yielding the result.
\end{proof}
Since the cost $c(x,y) = \frac{|x-y|}{1+|x-y|}$ is concave we also have $$K^*(X,Y) = \mathbb E(c(|X-Y|)) \leq c (\mathbb E(|X-Y|)) = \frac{E(|X-Y|)}{1+E(|X-Y|)} \, ,$$ so that
\begin{equation}\label{eqLPW1}
W_{LP}(\nu,\mu) \, \leq \, \frac{W_1(\nu,\mu)}{1+W_1(\nu,\mu)} \, .
\end{equation}
\medskip
\bibliographystyle{plain}
\bibliography{KLSLyap}

\end{document}